\documentclass{amsart}

\usepackage{stmaryrd}
\usepackage{enumerate}
\usepackage[pdftex]{graphicx}
\usepackage{amsmath}
\usepackage{amsthm}

\newtheorem{thm}{Theorem}[section]
\newtheorem{prop}[thm]{Proposition}
\newtheorem{lem}[thm]{Lemma}
\newtheorem{cor}[thm]{Corollary}

\theoremstyle{definition}
\newtheorem{definition}[thm]{Definition}

\theoremstyle{remark}
\newtheorem{notation}[thm]{Notation}
\newtheorem{remark}[thm]{Remark}

\numberwithin{equation}{section}

\newcommand{\id}{\mathrm{id}}

\newcommand{\Cmc}{\mathcal{C}} 
\newcommand{\Dmc}{\mathcal{D}} 
\newcommand{\Emc}{\mathcal{E}} 
\newcommand{\Fmc}{\mathcal{F}} 
\newcommand{\Gmc}{\mathcal{G}} 
\newcommand{\Hmc}{\mathcal{H}} 
\newcommand{\Imc}{\mathcal{I}} 
\newcommand{\Lmc}{\mathcal{L}} 
\newcommand{\Mmc}{\mathcal{M}} 
\newcommand{\Pmc}{\mathcal{P}} 
\newcommand{\Qmc}{\mathcal{Q}} 
 
\newcommand{\Tmc}{\mathcal{T}}

\newcommand{\Pbbb}{\mathbb{P}} 
\newcommand{\Rbbb}{\mathbb{R}} 
\newcommand{\Zbbb}{\mathbb{Z}} 
\newcommand{\PGL}{\mathrm{PGL}} 
\newcommand{\PSL}{\mathrm{PSL}} 
\newcommand{\SL}{\mathrm{SL}}

\begin{document}

\title{Deforming convex real projective structures}

\author{Anna Wienhard}
\address{Ruprecht-Karls Universit\"at Heidelberg, Mathematisches Institut, Im Neuenheimer Feld~288, 69120 Heidelberg, Germany
\newline HITS gGmbH, Heidelberg Institute for Theoretical Studies, Schloss-Wolfs\-brunnen\-weg 35, 69118 Heidelberg, Germany }
\email{wienhard@uni-heidelberg.de}
\thanks{AW was partially supported by the National Science Foundation under agreements DMS-1065919 and 0846408, by the Sloan Foundation, by the Deutsche Forschungsgemeinschaft, by the European Research Council under ERC-Consolidator grant 614733, and by the Klaus Tschira Foundation.}

\author{Tengren Zhang}
\address{Mathematics Department, California Institute of Technology, 1200 East California Boulevard, Mail Code 253-37, Pasadena, CA 91125}
\email{tengren@caltech.edu}
\thanks{}

%
%
%
\maketitle

\section{Introduction} 
Let $S$ be a closed connected orientable surface of genus $g\geq2$. A convex real projective structure on $S$ is a locally homogeneous $\Rbbb\Pbbb^2$-structure on $S$ which induces a diffeomorphism of $S$ with a manifold $M = \Omega/\Gamma$, where $\Omega \subset \Rbbb\Pbbb^2$ is a convex domain, and $\Gamma$ is a discrete group of projective transformations which preserve $\Omega$. 
Any hyperbolic structure on $S$ gives rise to a convex real projective structure on $S$, by taking $\Omega$ to be the Klein-Beltrami model of the hyperbolic plane in $\Rbbb\Pbbb^2$. The deformation space $\mathcal{C}(S)$ of convex real projective structures thus contains the Fricke-Teichm\"uller space $\mathrm{Hyp}(S)$ of hyperbolic structures on $S$. 
The study of convex real projective structures on $S$ has been pioneered by Bill Goldman, who rigorously defined the deformation space of convex real projective structures on $S$ and showed that it is a cell of dimension  $16g-16$ \cite{Goldman_convex}. He proved this result by giving a precise parametrization of the space of convex real projective structures associated to a pair of pants decomposition of the surface, involving two length and two twist parameters for each curve in the pants decomposition and two internal parameters for each pair of pants. 
The holonomy map provides an  embedding $\mathrm{hol} : \mathcal{C}(S) \rightarrow \mathrm{Hom}(\pi_1(S), \PGL(3,\Rbbb))/\PGL(3,\Rbbb)$. Choi-Goldman \cite{ChoiGoldman} showed that the image of $\Cmc(S)$ is the Hitchin component $\mathrm{Hit}_3(S) \subset \mathrm{Hom}(\pi_1(S), \PGL(3,\Rbbb))/\PGL(3,\Rbbb)$, which was introduced by Hitchin in \cite{Hitchin}. The space of convex real projective structures thus provides the first example of a higher Teichm\"uller space, and the study of convex real projective structures has often been a model case for the study of more general Hitchin components \cite{Goldman_convex, Li, Labourie_convex, Loftin_convex, FockGoncharov_convex, Zhang_convex}.

A simple closed curve on $S$ gives rise to a twist flow on $\mathrm{Hyp}(S)$. More generally, given a pair of pants decomposition of $S$, the twist flows along the $3g-3$ curves in the pants decomposition give rise to $3g-3$ pairwise commuting flows on $\mathrm{Hyp}(S)$ where $g$ is the genus of $S$.
Generalizations of these twist flows have been defined by Goldman \cite{Goldman_twist} in for representations of $\pi_1(S)$ into general reductive Lie groups, and described  very explicitly as shear and bulging deformations for convex real projective structures \cite{Goldman_bulging}. In the Fricke-Teichm\"uller space the $3g-3$ pairwise commuting twist flows fill out a half dimensional space ($\mathrm{dim}{\mathrm Hyp}(S) = 6g-6$). For convex real projective structures however, the shear and bulging flows associated to the $3g-3$ pants curves  give only rise to $6g -6$ pairwise commuting flows in the the $16g-16$-dimensional space of convex real projective structures. Moreover it is easy to see that the twist flows do not change the $4g-4$ internal parameters associated to the pants in the pair of pants decomposition. 

In this article we introduce two new flows on the space of convex real projective structures, which are described by explicit deformations of the internal parameters associated to each pair of pants in a pants decomposition. 
We call these flows the eruption flow and the internal bulging flow associated to the pants. The eruption flows associated to the $2g-2$ pairs of pants commute with the $6g-6$ generalized twist flows associated to the curves in the pants decomposition. Thus we obtain the following 
\begin{thm}\label{mainthm}
Let $S$ be a closed oriented surface of genus $g\geq 2$, and  $\Cmc(S)$ the space of convex real projective structures on $S$. Let $\Pmc$ be a pair of pants decomposition of $S$ and $\Tmc$ an ideal triangulation of $S$ adapted to $\Pmc$. Then the $6g-6$ shearing and bulging flows associated to the $3g-3$ pants curves and the $2g-2$ eruption flows associated to the $2g-2$ pairs of pants give rise to a half-dimensional family of commuting flows on $\Cmc(S)$. 
\end{thm}

In a forthcoming paper, joint with Zhe Sun, we will extend Theorem~\ref{mainthm} to Hitchin components for $\PGL(n,\Rbbb)$, and discuss the structure of this family of flows with respect to the symplectic structure on $\mathrm{Hom}(\pi_1(S), \PGL(n,\Rbbb))/\PGL(n,\Rbbb)$. 

Twist flows along simple closed curves have been generalized by Thurston to earthquakes along measured laminations, leading to Thurston's celebrated earthquake theorem \cite{Kerckhoff, Thurston_earthquake}, that given any two points $X, X' \in \mathrm{Hyp}(S)$ there exists a unique measured lamination on $X$ such that $X'$ can be obtained from $X$ by a left earthquake along this lamination. 
It is a very interesting and challenging question whether there exists any generalization of the earthquake theorem in the context of higher Teichm\"uller spaces. This is wide open even for the space of convex real projective structures. Trying to develop an understanding of how a generalization of the earthquake theorem in the space of convex real projective structures could look like was our initial motivation to define the eruption flows we introduce in this article. 

Twist flows on $\mathrm{Hyp}(S)$ are intimitately linked with the cross ratio of four points in the boundary $\Rbbb\Pbbb^1$ of the hyperbolic plane. In the same way, the generalized twist flows on $\mathcal{C}(S)$ are closely related to the generalized cross ratios of quadruples of flags in $\Rbbb^3$, and the new eruption flows we defined are closely related to the triple ratio, a projective invariant of a triple of flags in $\Rbbb^3$. These invariants play an important role in the work of Fock-Goncharov \cite{FockGoncharov}, as well as in recent work of Bonahon-Dreyer \cite{BonahonDreyer1, BonahonDreyer2} and Zhang \cite{Zhang_convex, Zhang_internal} parametrizing the Hitchin components for $\PGL(n,\Rbbb)$. 

This paper is organized as follows. In Section 2, we recall properties of some classical projective invariants, namely the cross ratio and triple ratio. Then in Section 3, we recall the work of Fock-Goncharov, who used these projective invariants to give a parameterization of the space of $n$-tuples of positive flags in $\Rbbb^3$, and describe the relationship between such $n$-tuples of flags and suitably nested $n$-gons. Using this, we describe the eruption, shearing and bulging flows in the context of $n$-tuples of positive flags. We then extend these flows to the setting of marked strictly convex domains with $C^1$ boundary in Section 4. In Section 5, we recall some basic facts about convex $\Rbbb\Pbbb^2$ structures on surfaces. Finally in Section 6, we use the elementary eruption, shearing and bulging flows from Section 4 to define eruption, shearing and internal bulging flows on $\Cmc(S)$.

\section{Projective invariants in $\Rbbb\Pbbb^2$} 
In this section, we give some background on projective geometry, and recall the definition of some projective invariants, which play an important role throughout the article: the cross ratio and the triple ratio. 

We begin by defining a notion of genericity of points in $\Rbbb\Pbbb^2$.
\begin{definition}
An $n$-tuple of points $p_1,\dots,p_n\in\Rbbb\Pbbb^2$ is generic if no three of the points lie in a projective line. Denote the set of generic $n$-tuples of points in $\Rbbb\Pbbb^2$ by $\Pmc_n$.
\end{definition}

One can easily verify that $\PGL(3,\Rbbb)$ acts simply transitively on $\Pmc_4$. However, if we consider quadruples of (non-generic) pairwise distinct points $p_1, p_2, p_3, p_4$, which lie on a projective line, the projective group $\PGL(3,\Rbbb)$ does not act transitively, and the orbits are given by the projective cross-ratio. 

\subsection{Cross ratios} \label{sec:crossratio}
We describe the cross-ratio in the dual picture and consider first the following collection of projective lines in $\Rbbb\Pbbb^2$.

\begin{definition}
Let $\Lmc_n$ denote the set of pairwise distinct $n$-tuples $l_1,\dots,l_n\in(\Rbbb\Pbbb^2)^*$ for which there exist some $p\in\Rbbb\Pbbb^2$ so that $l_i(p)=0$ for all $i=1,\dots,n$.
\end{definition}

Each $l\in(\Rbbb\Pbbb^2)^*$ is a projective class of linear functionals in $\Rbbb^3$. By taking the kernels of these linear functionals, we obtain a natural identification between $(\Rbbb\Pbbb^2)^*$ and the space of projective lines in $\Rbbb\Pbbb^2$. With this identification, the condition $l(p)=0$ means that the point $p$ lies in the projective line $l$. Hence $\Lmc_4$ identifies with the set of quadruples of lines which intersect in a point. 
We will blur the distinction between projective classes of linear functionals on $\Rbbb^3$ and projective lines in $\Rbbb\Pbbb^2$; it should be clear from the context which we are referring to. 

Observe that $\PGL(3,\Rbbb)$ acts on $\Lmc_4$. but  not transitively. The $\PGL(3,\Rbbb)$-orbits in $\Lmc_4$ can be characterized by the following function, called the cross ratio.

\begin{definition}
The \emph{cross ratio} is the function $C:\Lmc_4\to\Rbbb\setminus\{0,1\}$ defined by
\[C:(l_1,l_2,l_3,l_4)\mapsto\frac{l_1(p_3)\cdot l_4(p_2)}{l_1(p_2)\cdot l_4(p_3)},\]
where $p_2,p_3\in\Rbbb\Pbbb^2\setminus\{p\}$ are points that lie on $l_2$ and $l_3$ respectively, and $p$ is the common intersection point of $l_1,\dots,l_4$.
\end{definition}

In the above formula for $C$, we choose covectors $\alpha_1$, $\alpha_4$ and vectors $v_2,v_3$ in the projective classes of $l_1$, $l_4$ and $p_2$, $p_3$ respectively to evaluate the pairings $l_i(p_j):=\alpha_i(v_j)$. It is easy to check that $C(l_1,l_2,l_3,l_4)$ does not depend on these choices of vectors and covectors, nor does it depend on the choice of $p_2$ or $p_3$. For our purposes, we will also use the notation 
\[C(l_1,p_2,p_3,l_4):=C(l_1,l_2,l_3,l_4)\] 
for points $p_2\in l_2$ and $p_3\in l_3$. The next proposition states some properties of the cross ratio, which are well-known and easily verified.

\begin{prop} Let $(l_1,l_2,l_3,l_4)\in\Lmc_4$. The following statements hold.
\begin{enumerate}
\item $\displaystyle C(l_1,l_2,l_3,l_4)=\frac{1}{C(l_1,l_3,l_2,l_4)}=1-C(l_2,l_1,l_3,l_4)=C(l_4,l_3,l_2,l_1)$.
\item For any $g\in\PGL(3,\Rbbb)$, $\displaystyle C(l_1,l_2,l_3,l_4)=C(g\cdot l_1,g\cdot l_2,g\cdot l_3,g\cdot l_4)$.
\item $C$ is surjective, and its level sets are the $\PGL(3,\Rbbb)$-orbits in $\Lmc_4$, each of which is isomorphic to $\PGL(3,\Rbbb)$ as $\PGL(3,\Rbbb)$-sets.
\end{enumerate}
\end{prop}

The following proposition, whose proof is an elementary computation, allows to define the cross ratio of four points in $\Rbbb\Pbbb^2$ that lie in a projective line. 

\begin{prop}
Let $p_1,p_2,p_3,p_4\in\Rbbb\Pbbb^2$ be four pairwise distinct points in a projective line $l$. Let $q,q'$ be a pair of distinct points that do not lie in $l$, and for all $i=1,\dots,4$, let $l_i$ (resp. $l_i'$) be the projective lines through $q$ and $p_i$ (resp. $q'$ and $p_i$) respectively. Then
\[C(l_1,l_2,l_3,l_4)=C(l_1',l_2',l_3',l_4').\]
\end{prop}

We set 
\[C(p_1,p_2,p_3,p_4)=C(l_1,l_2,l_3,l_4).\]

Recall that a \emph{properly convex domain} $\Omega\subset\Rbbb\Pbbb^2$ is an open subset such that  for any pair of points $p,q\in\Omega$, there is a projective line segment between $p$ and $q$ that lies in $\Omega$, and the closure of $\Omega$ in $\Rbbb\Pbbb^2$ does not contain any projective lines. A properly convex domain $\Omega$ is \emph{strictly convex} if $\partial\Omega$ does not contain any line segments. The cross ratio of four points in a projective line allows us to define the Hilbert-metric on properly convex domains: 

\begin{definition}\label{def:Hilbert}
Let $\Omega\subset\Rbbb\Pbbb^2$ be a properly convex domain. For any pair of points $p,q\in\Omega$, let $a,b\in\partial\Omega$ be the points so that $a,p,q,b$ lie on a projective line in $\Rbbb\Pbbb^2$ in that order. The \emph{Hilbert metric} is the function $d_\Omega:\Omega^2\to\Rbbb$ given by 
\[d_\Omega(p,q)=\log|C(a,p,q,b)|.\]
\end{definition} 

It is easy to verify that $d_\Omega$ defines a metric on $\Omega$, and the projective invariance of the cross ratio implies that $d_\Omega$ is invariant under any projective transformation that leaves $\Omega$ invariant. Also, if $\Omega$ is strictly convex, then $d_\Omega$ is uniquely geodesic, i.e. there is a unique geodesic of $d_\Omega$ between any two points in $\Omega$, which is the projective line segment between them. Furthermore, if $\Omega$ is strictly convex, then $(\Omega, d_\Omega)$ is a $\delta$-hyperbolic metric space for some $\delta$. 

\subsection{Triple ratios} In this section we describe the triple ratio, which is a  projective invariant of  triples of pairwise transverse flags.

\begin{definition}
A \emph{flag} is a pair $(p,l)\in\Rbbb\Pbbb^2\times(\Rbbb\Pbbb^2)^*$ so that $l(p)=0$. Two flags $(p_1,l_1)$, $(p_2,l_2)$ are \emph{transverse} if $l_1(p_2)\neq 0\neq l_2(p_1)$. Let $\Fmc_n$ denote the set of ordered, pairwise transverse $n$-tuple of flags. 
\end{definition}

Under the identification of $(\Rbbb\Pbbb^2)^*$ with the space of projective lines in $\Rbbb\Pbbb^2$, $(p_1,l_1)$ and $(p_2,l_2)$ are transverse if and only if $p_1$ does not lie in $l_2$ and $p_2$ does not lie in $l_1$. 

\begin{definition}
The \emph{triple ratio} is the function $T:\Fmc_3\to\Rbbb\setminus\{0\}$ given by
\[T:\big((p_1,l_1),(p_2,l_2),(p_3,l_3)\big)\mapsto\frac{l_1(p_2)\cdot l_2(p_3)\cdot l_3(p_1)}{l_1(p_3)\cdot l_3(p_2)\cdot l_2(p_1)}.\]
\end{definition}

Just as we did in the case of cross ratios, we choose covectors $\alpha_1,\alpha_2,\alpha_3$ and vectors $v_1,v_2,v_3$ in the projective classes $l_1,l_2,l_3$ and $p_1,p_2,p_3$ respectively to evaluate $l_i(p_j):=\alpha_i(v_j)$. It is easy to verify that the triple ratio does not depend on any of these choices. 
The $\PGL(3,\Rbbb)$-orbits in $\Fmc_3$ can be described as the level sets of the triple ratio. The next proposition states some easily verified properties of the triple ratio.

\begin{prop}\label{triple ratio level set}
Let $(F_1,F_2,F_3)\in\Fmc_3$. The following statements hold.
\begin{enumerate}
\item $\displaystyle T(F_1,F_2,F_3)=\frac{1}{T(F_1,F_3,F_2)}$.
\item For any $g\in\PGL(3,\Rbbb)$, $T(F_1,F_2,F_3)=T(g\cdot F_1,g\cdot F_2,g\cdot F_3)$.
\item The triple ratio is surjective, and its level sets are the $\PGL(3,\Rbbb)$-orbits in $\Fmc_3$, each of which is isomorphic to $\PGL(3,\Rbbb)$ as $\PGL(3,\Rbbb)$-sets.
\end{enumerate}
\end{prop}

\begin{figure}[ht]
\centering
\includegraphics[scale=0.6]{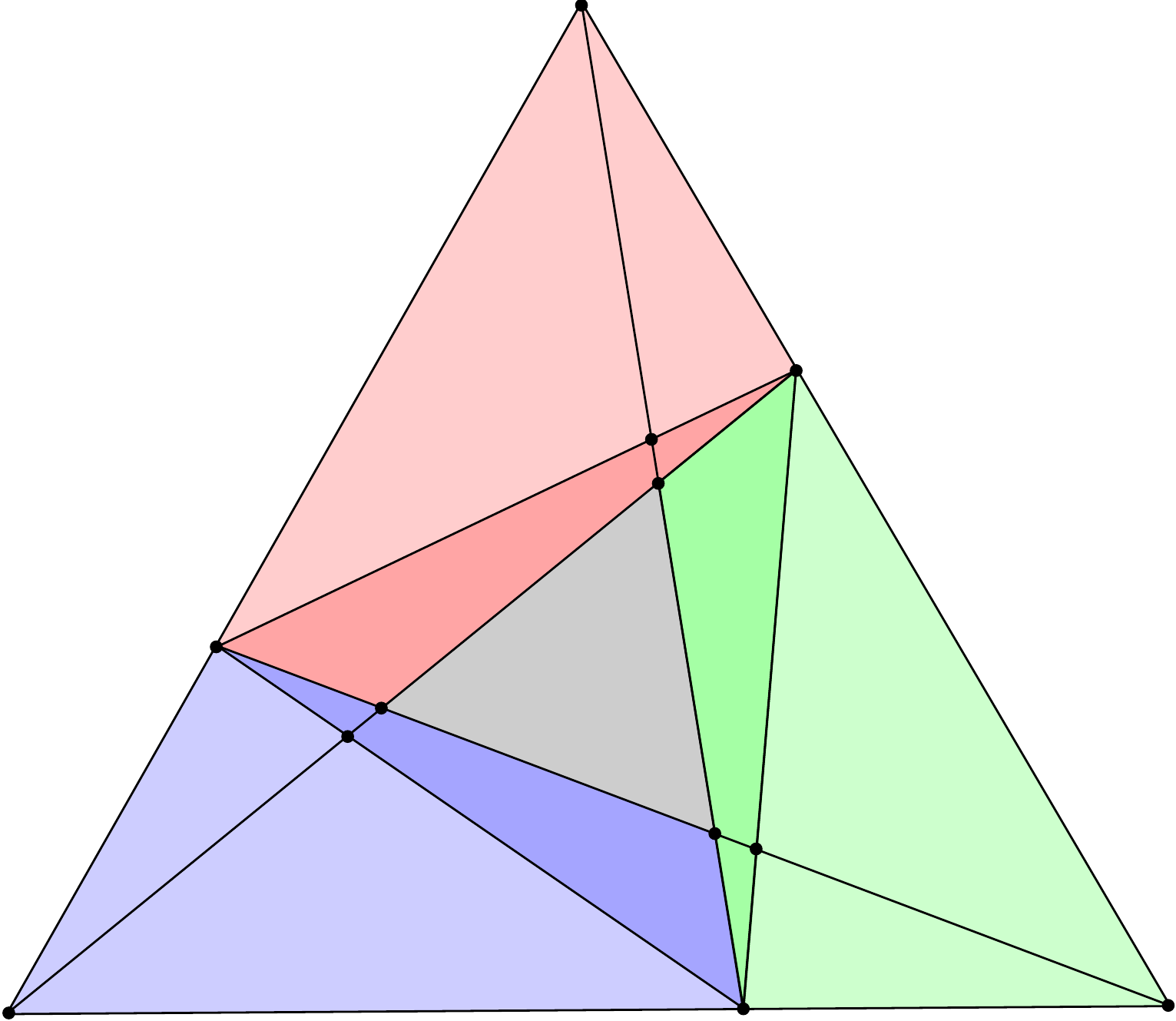}
\small
\put (-226, 85){$p_1$}
\put (-83, 147){$p_2$}
\put (-100, -3){$p_3$}
\put (-190, 145){$l_1$}
\put (-45, 80){$l_2$}
\put (-150, -5){$l_3$}
\put (-164, 36){$m_2$}
\put (-159, 120){$m_3$}
\put (-90, 90){$m_1$}
\put (0, 0){$q_1$}
\put (-273, -1){$q_2$}
\put (-140, 233){$q_3$}
\put (-67, 29){$w_1$}
\put (-215, 36){$w_2$}
\put (-137, 175){$w_3$}
\put (-93, 40){$r_1$}
\put (-190, 57){$r_2$}
\put (-129, 133){$r_3$}
\put (-115, 118){$u_1$}
\put (-115, 40){$u_2$}
\put (-185, 75){$u_3$}
\put (-75, 60){${\bf Q_1}$}
\put (-190, 20){${\bf Q_2}$}
\put (-160, 150){${\bf Q_3}$}
\put (-108, 100){${\bf T_1}$}
\put (-145, 42){${\bf T_2}$}
\put (-175, 91){${\bf T_3}$}
\put (-135, 75){${\bf T}$}
\caption{Decomposition of two nested triangles into elementary pieces}\label{notation picture}
\end{figure}
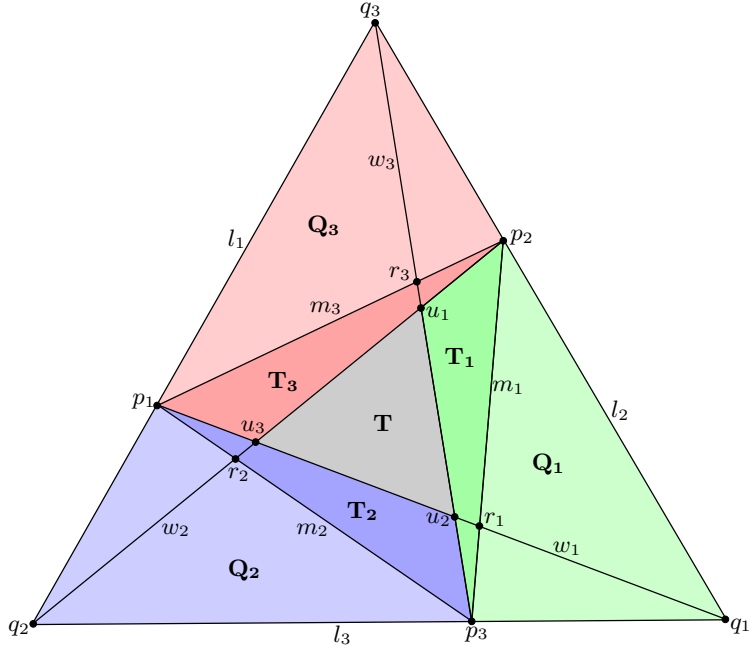

We will now describe a relationship between triple ratios and cross ratios, which allows us to give an new geometric interpretation of the triple ratio. 
To do so, we set up some notation. 

\begin{notation}\label{triple notation}
Let $\big((p_1,l_1),(p_2,l_2),(p_3,l_3)\big)\in\Fmc_3$. For $i,j,k=1,2,3$ so that $i,j,k$ are pairwise distinct, let $q_k:=l_i\cap l_j$ and let $m_k$ be the projective line through $p_i$ and $p_j$. Also, let $w_k$ be the line through $p_k$ and $q_k$, let $t_k:=l_k\cap m_k$, and let $r_k:=w_k\cap m_k$. Finally, let $u_k:=w_i\cap w_j$. (See Figure \ref{notation picture}.) Note that for the rest of this article, arithmetic in the subscripts used in this notation are done modulo $3$. 
\end{notation}

Observe that for all $i=1,2,3$, $p_i$, $u_{i-1}$, $u_{i+1}$, $q_i$ lie on a common projective line.

\begin{prop}\label{cross ratio and triple ratio}
Let $(F_1,F_2,F_3)=\big((p_1,l_1),(p_2,l_2),(p_3,l_3)\big)\in\Fmc_3$. For all $i=1,2,3$, we have
\[C(p_i,u_{i-1},u_{i+1},q_i)=T(F_1,F_3,F_2),\]
where $u_{i-1}$, $u_{i+1}$ and $q_i$ are as defined in the paragraph above.
\end{prop}

\begin{proof}
Choose coordinates so that $l_1=[1:0:0]$, $l_2=[0:1:0]$, $l_3=[0:0:1]$, $p_1=[0:b_1,c_1]^T$, $p_2=[a_2:0:c_2]^T$, $p_3=[a_3:b_3:0]^T$. Then $q_1=[1:0:0]^T$, $q_2=[0:1:0]^T$ and $q_3=[0:0:1]^T$. Also, one can compute that $w_1=[0:c_1:-b_1]$, $w_2=[-c_2:0:a_2]$ and $w_3=[b_3:-a_3:0]$, which implies that $u_1=[1:\frac{b_3}{a_3}:\frac{c_2}{a_2}]^T$, $u_2=[\frac{a_3}{b_3}:1:\frac{c_1}{b_1}]^T$ and $u_3=[\frac{a_2}{c_2}:\frac{b_1}{c_1}:1]^T$. 

With this, it is an easy computation to show that for all $i=1,2,3$, 
\[C(p_i,u_{i-1},u_{i+1},q_i)=\frac{b_1c_2a_3}{a_2b_3c_1}.\]
On the other hand, one can also compute that
\[T(F_1,F_3,F_2)=\frac{b_1c_2a_3}{a_2b_3c_1}.\]
\end{proof}

Using the sign of the triple ratio, we can pick out a particular subset of $\Fmc_n$, \cite{FockGoncharov}.

\begin{definition}
A $n$-tuple $(F_1,\dots,F_n)\in\Fmc_n$ is \emph{positive} if $T(F_i,F_j,F_k)>0$ for all triples $F_i<F_j<F_k<F_i$ in the cyclic order on $(F_1,\dots,F_n)$. Denote the set of positive triples in $\Fmc_n$ by $\Fmc_n^+$.
\end{definition}

The positivity of a triple $(F_1,F_2,F_3)=\big((p_1,l_1),(p_2,l_2),(p_3,l_3)\big)\in\Fmc_3$ can be interpreted in the following way. For $i,j,k=1,2,3$ that are pairwise distinct, let $t_k$ and $r_k$ be as defined in Notation \ref{triple notation}. The triple $(F_1,F_2,F_3)$ is positive if and only if $r_k$ and $t_k$ lie in distinct connected components of $\Rbbb\Pbbb^2\setminus(l_i\cup l_j)$. Equivalently, the triple ratio of $(F_1,F_2,F_3)=\big((p_1,l_1),(p_2,l_2),(p_3,l_3)\big)\in\Fmc_3$ is positive if there is a triangle $\Delta\subset\Rbbb\Pbbb^2$ with vertices $p_1, p_2, p_3$ and a triangle $\Delta'\subset\Rbbb\Pbbb^2$ with edges $l_1, l_2, l_3$ so that $\Delta\subset\Delta'$. In Figure \ref{notation picture} we have $\Delta = T \cup T_1\cup T_2 \cup T_3$ and $\Delta' = T \cup Q_1\cup Q_2 \cup Q_3$ (see Section~\ref{polygonsection} for a precise description of the correspondence of $\Fmc_n^+$ and suitably nested polygons). 

With this notation Proposition~\ref{cross ratio and triple ratio} implies the following
\begin{cor}\label{cor:tripleratio}
Let $(F_1,F_2,F_3)=\big((p_1,l_1),(p_2,l_2),(p_3,l_3)\big)\in\Fmc_3$. Let $\Delta'$ be the triangle with edges $l_1, l_2, l_3$ and $T$  as in Figure~\ref{notation picture}. 
Then $\log T(F_1,F_3,F_2)$ is the Hilbert length of the side of the triangle $T$ with respect to the proper convex set $\Delta'$. 
In particular $T$ is an equilateral triangle with respect to this Hilbert metric. 
\end{cor}

\begin{remark}\label{rem:assembly}
\begin{enumerate}
\item Note that if $(F_1,F_2,F_3) = \big((p_1,l_1),(p_2,l_2),(p_3,l_3)\big)\in\Fmc_3$ is a triple of flags which arises by taking three points $p_1, p_2, p_3$  on the boundary of a quadric, and $l_1, l_2, l_3$ the tangent to the quadric through these points, then $\log T(F_1,F_3,F_2) = 0$. 
\item Corollary~\ref{cor:tripleratio} allows us to interpret the triple ratio as instruction to assemble the configuration $\big((p_1,l_1),(p_2,l_2),(p_3,l_3)\big)$ out of the quadrilateral $Q_1, Q_2, Q_3$. Since $\PGL(3,\Rbbb)$ acts transitively on $\Pmc_4$, each of the quadrilaterals $Q_1, Q_2, Q_3$ in Figure \ref{notation picture} is equivalent up to projective transformation, and the triple ratio can be seen as the gluing parameter. This gives an interpretation similar to the interpretation of the hyperbolic cross ratio function as the gluing parameter for assembling an ideal hyperbolic quadrilateral out of two ideal hyperbolic triangles.  
We will make use of this point of view in Section~\ref{deforming a triple of flags}.
\end{enumerate}
\end{remark}

\section{Flows on the space of nested polygons} 

In this section, we define some flows on $\Fmc_n^+$ which we call the shearing flow, the bulging flow and the eruption flow. 

The eruption flow is naturally defined on $\Fmc_3^+$ and continuously changes the triple ratio. It is based on viewing the triple ratio as gluing parameters to assemble the triple of flags $(F_1, F_2, F_3) $ out of three projective quadilaterals $Q_1, Q_2, Q_3$ as described in Remark~\ref{rem:assembly}. The shearing and bulging flows are naturally defined on $\Fmc_4^+$. They are based on considering the parameters to glue a quadruple $(F_1, F_2, F_3, F_4)$ out of two triples $(F_1, F_2, F_3)$ and $(F_1, F_3, F_4)$. Decomposing an $n$-tuple of flags in $\Fmc_n^+$ into successive triples allows to extend the shearing, bulging and eruption flows to $\Fmc_n^+$. 
In order to describe these flows it is useful to identify points  in $\Fmc_n^+$ with nested polygons. 

\subsection{Suitably nested, labelled polygons}\label{polygonsection}
A \emph{polygon} in $\Rbbb\Pbbb^2$, is a simply connected, properly convex, compact set in $\Rbbb\Pbbb^2$ whose boundary is a union of finitely many projective line segments. These projective line segments are the \emph{edges} of the polygon, and the endpoints of these edges are the \emph{vertices} of the polygon. 

\begin{definition}\
\begin{enumerate}
\item A \emph{labelled polygon} is a polygon equipped with an ordering on its vertices, so that the successor of any vertex $v$ in this ordering is connected to $v$ by an edge. For any labelled polygon $N$, let $p_1(N),\dots,p_n(N)$ denote the vertices of $N$, enumerated in according to the order on the vertices. Also, let $e_1(N),\dots,e_n(N)$ denote the edges of $N$, enumerated so that the endpoints of $e_i(N)$ are $p_i(N)$ and $p_{i+1}(N)$ for all $i=1,\dots,n$. (Here, $p_{n+1}(N):=p_1(N)$.)
\item A pair $(N,N')$ of labelled $n$-gons are \emph{suitably nested} if $N\subset N'$, and $p_i(N)$ lies in the interior of $e_i(N')$ for all $i=1,\dots,n$.
\end{enumerate}
\end{definition}

The next proposition is well-known, and relates $n$-tuples of positive flags in $\Rbbb\Pbbb^2$ to suitably nested, labelled $n$-gons. (See Figure \ref{polygonflags}.)

\begin{prop}\cite[Theorem~2.2]{FockGoncharov_convex}\label{polygon}
Let $\big((p_1,l_1),\dots,(p_n,l_n)\big)=F\in\Fmc_n$. $F\in\Fmc_n^+$ if and only if there is a (necessarily unique) pair of suitably nested, labelled $n$-gons $(N,N')$ so that $p_i(N)=p_i$ and $e_i(N')\subset l_i$.
\end{prop}

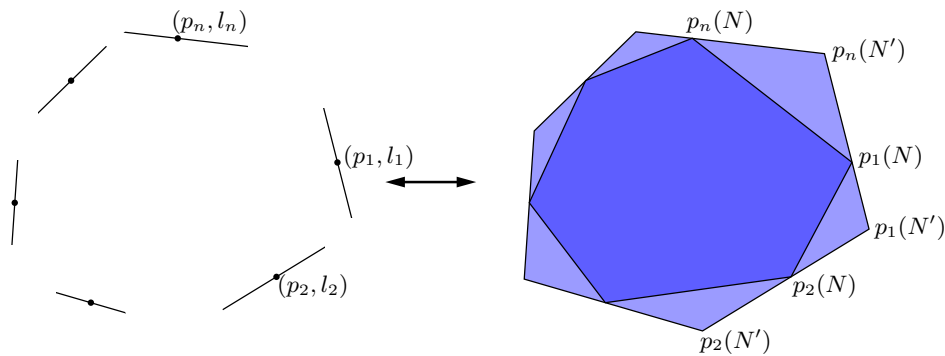
\begin{figure}[ht]
\centering
\includegraphics[scale=0.6]{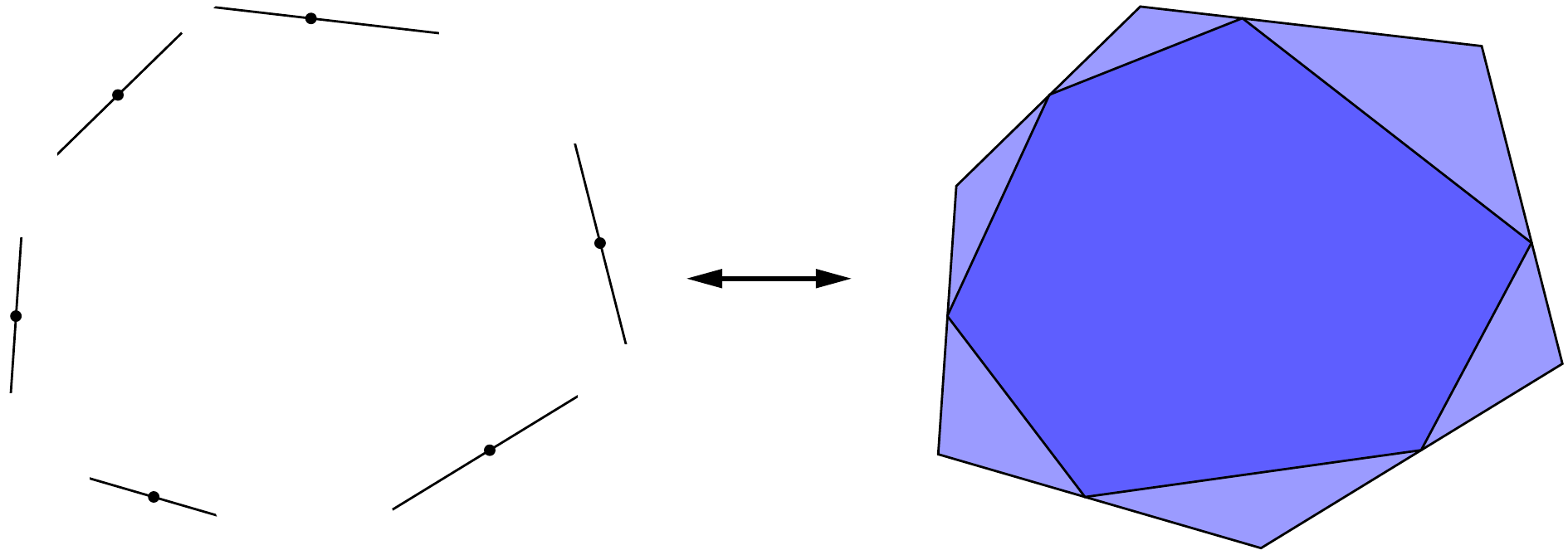}
\small
\put (-200, 65){$(p_1,l_1)$}
\put (-225, 16){$(p_2,l_2)$}
\put (-265, 116){$(p_n,l_n)$}
\put (-5, 65){$p_1(N)$}
\put (-30, 16){$p_2(N)$}
\put (-70, 116){$p_n(N)$}
\put (1, 38){$p_1(N')$}
\put (-65, -5){$p_2(N')$}
\put (-15, 105){$p_n(N')$}
\caption{$n$-tuples of flags in $\Fmc_n$ and suitably nested pairs of $n$-gons}\label{polygonflags}
\end{figure}

Proposition \ref{polygon} gives a canonical homeomorphism between $\Fmc_n^+$ and the space of suitably nested, labelled $n$-gons. As a consequence, we will henceforth blur the distinction between $n$-tuples of positive flags in $\Rbbb\Pbbb^2$ and pairs of suitably nested, labelled $ n$-gons.

We describe an explicit parametrization of $\Fmc_n^+$ in terms of the projective invariants we introduced, cross ratios and triple ratios. For this consider a pair of suitably nested labelled polygons $(N,N')=\big((\overline{p}_1,\overline{l}_1),\dots,(\overline{p}_n,\overline{l}_n)\big)\in\Fmc_n^+$ and choose a triangulation $\Tmc$ of $N$, so that the set of vertices of the triangulation is $\{\overline{p}_1,\dots,\overline{p}_n\}$, i.e. the set of vertices of $N$. Note that this induces a triangulation of the labelled $n$-gon $M$ for every $(M,M')\in\Fmc_n^+$. Let $I_\Tmc$ denote the set of internal edges of $\Tmc$ and $\Theta_\Tmc$ denote the set of triangles of $\Tmc$.

Now, let $i,j\in\{1,\dots,n\}$ so that $\overline{p}_i,\overline{p}_j$ are endpoints of some internal edge $a_{i,j}\in I_\Tmc$. Then let $k,k',\in\{1,\dots,n\}$ so that $i<k<j<k'<i$ in the obvious cyclic ordering on $\{1,\dots,n\}$, and $\overline{p}_i,\overline{p}_j,\overline{p}_k$ and $\overline{p}_i,\overline{p}_j,\overline{p}_{k'}$ are vertices of the two triangles in $\Theta_\Tmc$ that have $e_{i,j}$ as a common edge. For any $F=\big((p_1,l_1),\dots,(p_n,l_n)\big)\in\Fmc_n^+$, define 
\[\sigma_{i,j}(F):=\log\big(-C(l_i,p_k,p_k',\overline{p_ip_j})\big),\]
where $\overline{p_ip_j}$ denotes the projective line through $p_i$ and $p_j$ and $C$ is the cross ratio introduced in Section~\ref{sec:crossratio} . Given Proposition \ref{polygon}, it is easy to check that $\sigma_{i,j}(F)$ is well-defined. This associates to every internal edge $e_{i,j}\in I_\Tmc$ two functions $\sigma_{i,j},\sigma_{j,i}:\Fmc_n^+\to\Rbbb$. Since these functions are projective invariants, they descend to functions $\sigma_{i,j},\sigma_{j,i,}:\PGL(3,\Rbbb)\backslash\Fmc_n^+\to\Rbbb$.

Similarly, let $i,j,k\in\{1,\dots,n\}$ so that $i<j<k<i$, and there is a triangle $T_{i,j,k}\in\Theta_\Tmc$ with vertices $\overline{p}_i$, $\overline{p}_j$ and $\overline{p}_k$. Then define
\[\tau_{i,j,k}(F):=\log T\big((p_i,l_i),(p_j,l_j),(p_k,l_k)\big).\]
This associates to every triangle in $T_{i,j,k}\in\Theta_\Tmc$ the function $\tau_{i,j,k}:\Fmc_n^+\to\Rbbb$, which descends to a function $\tau_{i,j,k}:\PGL(3,\Rbbb)\backslash\Fmc_n^+\to\Rbbb$. The next proposition tells us that the projective invariants $\sigma_{i,j}$ and $\tau_{i,j,k}$ can be used to completely understand $\Fmc_n^+$. 
\begin{prop}\cite[Theorem 2.2]{FockGoncharov_convex}\label{finite parameterization}
The map 
\[\big((\sigma_{i,j},\sigma_{j,i})_{e_{i,j}\in I_\Tmc},(\tau_{i,j,k})_{T_{i,j,k}\in\Theta_\Tmc}\big):\PGL(3,\Rbbb)\backslash\Fmc_n^+\to\Rbbb^{2I_\Tmc+\Theta_\Tmc}\] 
is a homeomorphism.
\end{prop}

\begin{figure}[ht]
\centering
\includegraphics[scale=0.6]{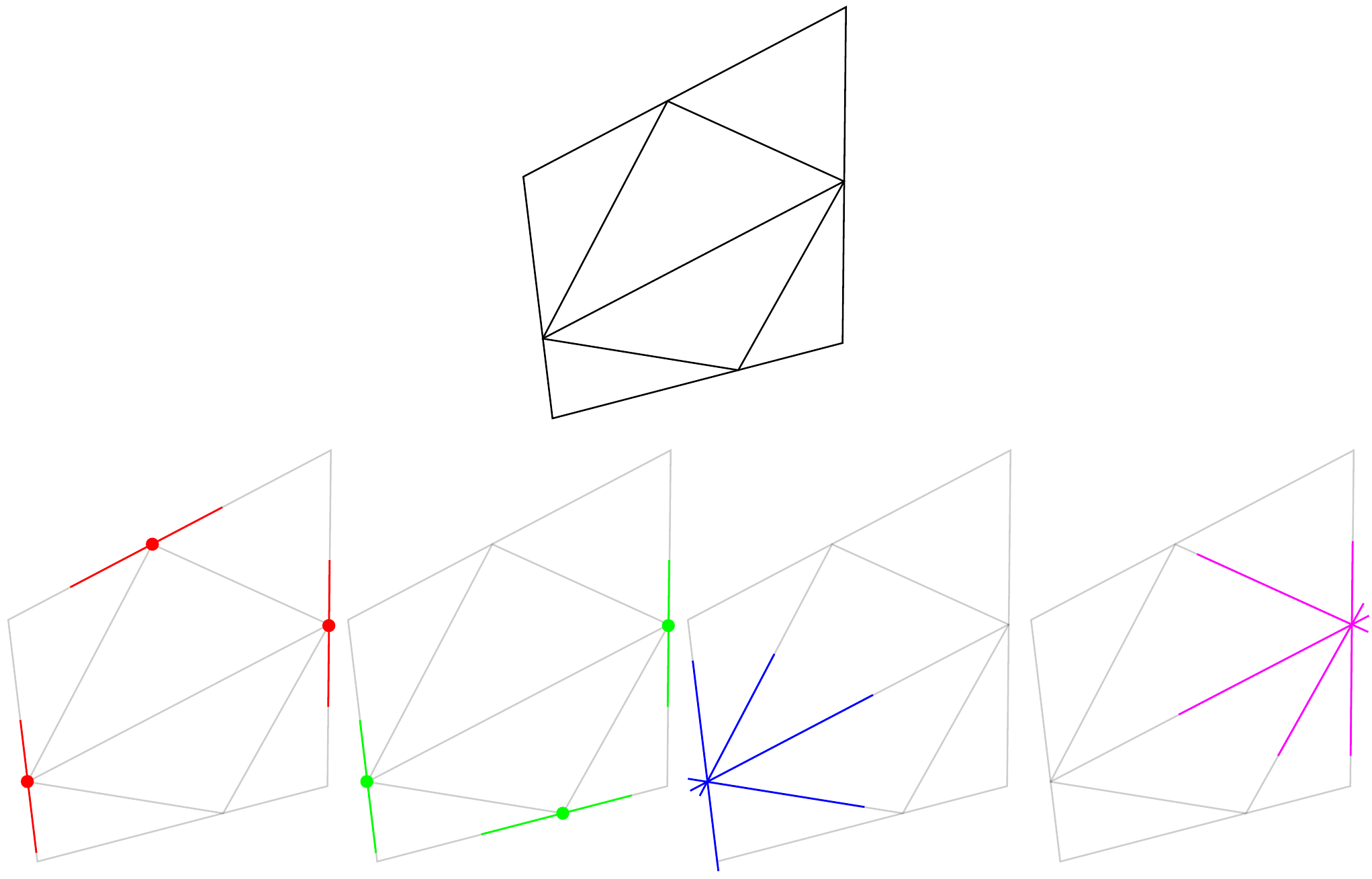}
\small
\put (-243, 137){$(p_1,l_1)$}
\put (-206, 204){$(p_2,l_2)$}
\put (-135, 177){$(p_3,l_3)$}
\put (-165, 123){$(p_4,l_4)$}
\put (-320, -5){$\tau_{1,2,3}$}
\put (-230, -5){$\tau_{3,4,1}$}
\put (-140, -5){$\sigma_{1,3}$}
\put (-50, -5){$\sigma_{3,1}$}
\caption{The parameterization of $\Fmc_4^+$ given by Proposition \ref{finite parameterization}}\label{finiteparameterization}
\end{figure}
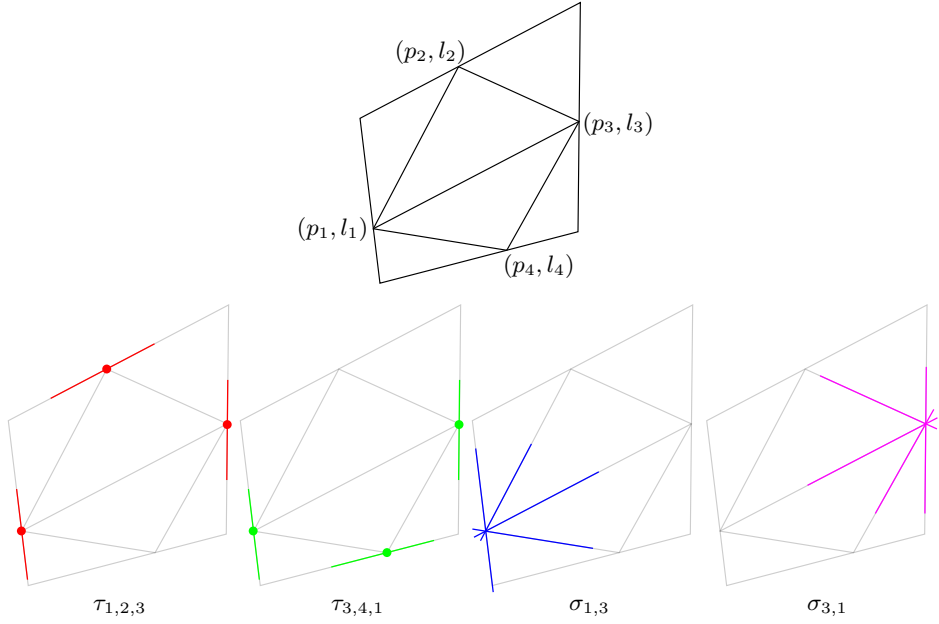

We will make use of the following proposition, whose proof we leave to the reader. 
\begin{prop}\label{finite convergence}
Let $\left\{F^{(s)}=\big((p^{(s)}_1,l^{(s)}_1),\dots,(p^{(s)}_n,l^{(s)}_n)\big)\right\}_{s=1}^\infty$ be a sequence in $\Fmc_n^+$ with the following properties:
\begin{itemize}
\item For all $e_{i,j}\in I_\Tmc$, $\sigma_{i,j}(F^{(s)})$ and $\sigma_{j,i}(F^{(s)})$ both converge to positive real numbers as $s\to\infty$,
\item For all $T_{i,j,k}\in\Theta_\Tmc$, $\tau_{i,j,k}(F^{(s)})$ converges to a positive real number as $s\to\infty$,
\item For $i=1,2,3$, the sequences $\{p_i^{(s)}\}_{s=1}^\infty$ converge in $\Rbbb\Pbbb^2$,
\item For $i=1,2$, the sequences $\{l_i^{(s)}\}_{s=1}^\infty$ converge $(\Rbbb\Pbbb^2)^*$,
\item For $i=1,2$, $j=1,2,3$, $i\neq j$, $\lim_{s\to\infty}p_i^{(s)}$ does not lie in $\lim_{s\to\infty}l_j^{(s)}$.
\end{itemize}
Then the sequence $\{F^{(s)}\}_{s=1}^\infty$ converges in $\Fmc_n^+$.
\end{prop}

\subsection{Deforming a triple of flags}\label{deforming a triple of flags}
Let $(F_1, F_2, F_3) = \big((p_1,l_1),(p_2,l_2),(p_3,l_3)\big) \in \Fmc^+_3$ and $(\Delta, \Delta')$ the corresponding pair of suitably nested, labelled triangles. 
We will now specify a systematic way to obtain three labelled quadrilaterals from $(\Delta,\Delta')\in\Fmc_3^+$. For $i=1,2,3$, let $q_i$, $r_i$, $u_i$, and $m_i$ be as defined in Notation \ref{triple notation}. For $i=1,2,3$, let $Q_i$ be the labelled quadrilateral in $\Delta'$ with vertices $q_i,p_{i-1},u_i,p_{i+1}$. Note that $\Delta'=Q_1\cup Q_2\cup Q_3\cup T$, where $T$ is the labelled triangle $T$ whose vertices are $u_1$, $u_2$, $u_3$. In each $Q_i$, let $T_i\subset Q_i$ be the labelled triangle with vertices $u_i$, $p_{i+1}$, $p_{i-1}$. Note then that $\Delta=T_1\cup T_2\cup T_3\cup T$. (See Figure \ref{notation picture}.)

Conversely, given three labelled quadrilaterals in $\Pmc_4$, we can use Proposition~\ref{cross ratio and triple ratio} to assemble them together to obtain any pair of suitably nested, labelled triangles $(\Delta,\Delta')\in\Fmc_3^+$. With this we interpret the triple ratio of $(F_1,F_2,F_3)\in\Fmc_3^+$ as instruction to assemble $(\Delta,\Delta')$ from a triple of labelled quadrilaterals in $\Pmc_4$. Since there is a unique labelled quadrilateral in $\Rbbb\Pbbb^2$ up to projective transformations, by deforming the corresponding ``assembly instructions", we obtain a path of deformations of triples of flags that are projectively non-equivalent.

More explicitly, let $v_1$, $v_2$, $v_3$ be non-zero vectors that span $q_1$, $q_2$, $q_3$ respectively. Then let $g_1(t),g_2(t),g_3(t)\in\PGL(3,\Rbbb)$ that have the following matrix representations when written in the basis $\{v_1,v_2,v_3\}$:
\[
g_1(t):=\left[\begin{array}{ccc}
1&0&0\\
0&e^\frac{t}{3}&0\\
0&0&e^{-\frac{t}{3}}
\end{array}\right],\,\,
g_2(t):=\left[\begin{array}{ccc}
e^{-\frac{t}{3}}&0&0\\
0&1&0\\
0&0&e^\frac{t}{3}
\end{array}\right],\,\,
g_3(t):=\left[\begin{array}{ccc}
e^\frac{t}{3}&0&0\\
0&e^{-\frac{t}{3}}&0\\
0&0&1
\end{array}\right].
\]
It is an easy computation to check that for all $i=1,2,3$, $g_{i-1}(t)\cdot p_i=g_{i+1}(t)\cdot p_i$ lies on $l_i$. Furthermore,  $u_i(t) := g_i(t)\cdot u_i$ lies on the line through $q_{i-1}$ and $g_i(t)\cdot p_{i-1}$. These together imply that there is a unique labelled triangle $T(t)$ (with vertices $u_1(t), u_2(t), u_3(t)$) so that $\big(\Delta(t),\Delta'(t)\big)\in\Fmc_3^+$, where
\[\Delta'(t):=\big(g_1(t)\cdot Q_1\big)\cup \big(g_2(t)\cdot Q_2\big)\cup \big(g_3(t)\cdot Q_3\big)\cup T(t)\]
and 
\[\Delta(t):=\big(g_1(t)\cdot T_1\big)\cup \big(g_2(t)\cdot T_2\big)\cup \big(g_3(t)\cdot T_3\big)\cup T(t).\]

\begin{definition}
The \emph{eruption flow on $\Fmc_3^+$} is the flow $\epsilon_t:\Fmc_3^+\to\Fmc_3^+$ defined by $\epsilon_t:(\Delta,\Delta')\mapsto\big(\Delta(t),\Delta'(t)\big)$.
\end{definition}

Let $\big(F_1(t), F_2(t), F_3(t)\big)\in\Fmc_3^+$ be the triple of flags corresponding to $\big(\Delta(t),\Delta'(t)\big)$, then $T\big(F_1(t), F_2(t), F_3(t)\big)=e^t\cdot T(F_1,F_2,F_3)$. In particular, the eruption flow on $\Fmc_3^+$ preserves the foliation of $\Fmc_3^+$ by its $\PGL(3,\Rbbb)$-orbits. This implies that the eruption flow descends to a smooth flow on the $1$-dimensional manifold $\Fmc_3^+/\PGL(3,\Rbbb)$.

\begin{remark}
The name eruption flow arise from imagining the triangle $\Delta$ as a volcano, with $T$ being the opening of the volcano. Thus applying the eruption flow with $t>0$ let's this volcano erupt more. 
\end{remark}

\subsection{Deforming a quadruple of flags} 
Given a quadruple $\big(F_1, F_2, F_3, F_4\big) \in \Fmc_4^+$, let $(N,N')$ be the associated pair of nested quadrilaterals. We obtain two triples $(F_1, F_3, F_2), (F_1, F_3, F_4) \in \Fmc_3^+$ and a decomposition of $N$ into two triangles $N_R, N_L$, which lie to the right, respectively left of the diagonal $a_{1,3}$ in $N$ with backward endpoint $p_1$ and forward endpoint $p_3$. Note that the diagonal $a_{1,3}$ decomposes $N'$ into two quadrilaterals. The shear and bulge flows associated to this decomposition of $N$ into the two triangles are defined as follows. 

Let $v_1, v_3, v_{1,3}\in\Rbbb^3$ be non-zero vectors that span $p_1$, $p_3$, and the intersection point $l_1\cap l_3$ respectively. Define $s(t),b(t)\in\PGL(3,\Rbbb)$ to be the projective transformations that are represented by the matrices

\[
s(t):=\left[\begin{array}{ccc}
e^\frac{t}{2}&0&0\\
0&1&0\\
0&0&e^{-\frac{t}{2}}
\end{array}\right],\,\,
b(t):=\left[\begin{array}{ccc}
e^{-\frac{t}{6}}&0&0\\
0&e^\frac{t}{3}&0\\
0&0&e^{-\frac{t}{6}}
\end{array}\right]\]
in the basis $\{v_1,v_{1,3},v_3\}$.

Let 
\begin{eqnarray*}
N(t):=\big(s(t)\cdot N_L\big)\cup\big(s(-t)\cdot N_R\big),\\
N'(t):=\big(s(t)\cdot N_L'\big)\cup\big(s(-t)\cdot N_R'\big),\\
M(t):=\big(b(t)\cdot N_L\big)\cup\big(b(-t)\cdot N_R\big),\\
M'(t):=\big(b(t)\cdot N_L'\big)\cup\big(b(-t)\cdot N_R'\big).
\end{eqnarray*}
Observe that $s(t)$ and $b(t)$ both fix $p_1$, $p_3$ and stabilize $a_{1,3}$, $l_1$, $l_3$. This implies that the pairs $\big(N(t),N'(t)\big)$ and $\big(M(t),M'(t)\big)$ are suitably nested, labelled quadrilaterals.

\begin{definition}\
\begin{enumerate}
\item The \emph{shearing flow on $\Fmc_4^+$ associated to $a_{1,3}$} is the flow $(\psi)_t:\Fmc_4^+\to\Fmc_4^+$ defined by $(\psi)_t:(N,N')\mapsto\big(N(t),N'(t)\big)$.
\item The \emph{bulging flow on $\Fmc_4^+$ associated to $a_{1,3}$} is the flow $(\beta)_t:\Fmc_4^+\to\Fmc_4^+$ defined by $(\beta)_t:(N,N')\mapsto\big(M(t),M'(t)\big)$.
\end{enumerate}
\end{definition}

If $\big(F_1(t),F_2(t),F_3(t),F_4(t)\big)\in\Fmc_4^+$ is the quadruple of flags corresponding to $\big(N(t),N'(t)\big)$, then one can compute that 
\[C\big(l_1(t),p_2(t),p_4(t),\overline{p_1(t)p_3(t)}\big)=e^{-t}\cdot C\big(l_1,p_2,p_4,\overline{p_1p_3}\big),\]
\[C\big(l_3(t),p_4(t),p_2(t),\overline{p_1(t)p_3(t)}\big)=e^{-t}\cdot C\big(l_3,p_4,p_2,\overline{p_1p_3}\big).\] 
A similar computation also proves that if $\big(F_1(t),F_2(t),F_3(t),F_4(t)\big)\in\Fmc_4^+$ is the quadruple of flags corresponding to $\big(M(t),M'(t)\big)$, then 
\[C\big(l_1(t),p_2(t),p_4(t),\overline{p_1(t)p_3(t)}\big)=e^t\cdot C\big(l_1,p_2,p_4,\overline{p_1p_3}\big),\] 
\[C\big(l_3(t),p_4(t),p_2(t),\overline{p_1(t)p_3(t)}\big)=e^{-t}\cdot C\big(l_3,p_4,p_2,\overline{p_1p_3}\big).\]
Here, recall that for $i=1,\dots,4$, $F_i(t)=\big(p_i(t),l_i(t)\big)$ and $F_i=(p_i,l_i)$. Also, $\overline{p_1(t)p_3(t)}$ is the projective line through $p_1(t)$ and $p_3(t)$, and $\overline{p_1p_3}$ is the projective line through $p_1$ and $p_3$.

\subsection{Deforming an $n$-tuple of flags}
We now extend the eruption, shearing and bulging flows to pairwise commuting flows on $\Fmc_n^+$ and $\Fmc_n^+/\PGL(3,\Rbbb)$ so that the $\Rbbb^m$ action on $\Fmc_n^+/\PGL(3,\Rbbb)$ is transitive.

Let $(N,N')=\big((p'_1,l'_1),\dots,(p'_n,l'_n)\big)\in\Fmc_n^+$. Choose a triangulation of $N$ as we did in Section \ref{polygonsection}, so that the set of vertices of the triangulation is $\{p'_1,\dots,p'_n\}$. Let $a_{i,j}$ be any edge of this triangulation with endpoints $p'_i,p'_j$. The labelling of the vertices of $N$ induces an orientation on $N$, so $a_{i,j}$ cuts both $N$ and $N'$ into two labelled polygons. (These two polygons are always non-empty in the case of $N'$, and they are non-empty in the case of $N$ if and only if $a_{i,j}$ is not a boundary segment of $N$.) Let $N_L$, $N_R$, $N_L'$, $N_R'$ be the labelled polygons, so that $N_L$ and $N_L'$ lie on the left of $a_{i,j}$, $N_R$ and $N_R'$ lie on the right of $a_{i,j}$, $N_L\cup N_R=N$ and $N_L'\cup N_R'=N'$. 

Let $v_i, v_j,v_{i,j}$ be non-zero vectors that span $p'_i$, $p'_j$, $l_i'\cap l'_j$ respectively. Then let $s_{i,j}(t),b_{i,j}(t)\in\PGL(3,\Rbbb)$ be the projective transformations that are represented by the matrices $s(t)$, $b(t)$ as above with respect to the basis $\{v_i,v_{i,j},v_j\}$. Set 
\begin{eqnarray*}
N_{i,j}(t):=\big(s_{i,j}(t)\cdot N_L\big)\cup\big(s_{i,j}(-t)\cdot N_R\big),\\
N_{i,j}'(t):=\big(s_{i,j}(t)\cdot N_L'\big)\cup\big(s_{i,j}(-t)\cdot N_R'\big),\\
M_{i,j}(t):=\big(b_{i,j}(t)\cdot N_L\big)\cup\big(b_{i,j}(-t)\cdot N_R\big),\\
M_{i,j}'(t):=\big(b_{i,j}(t)\cdot N_L'\big)\cup\big(b_{i,j}(-t)\cdot N_R'\big).
\end{eqnarray*}
Then $s_{i,j}(t)$ and $b_{i,j}(t)$ both fix $p'_i$, $p'_j$ and stabilize $a_{i,j}$, $l'_i$, $l'_j$, and the pairs $\big(N_{i,j}(t),N_{i,j}'(t)\big)$ and $\big(M_{i,j}(t),M_{i,j}'(t)\big)$ are suitably nested, labelled $n$-gons.

\begin{definition}\
\begin{enumerate}
\item The \emph{shearing flow on $\Fmc_n^+$ associated to $a_{i,j}$} is the flow $(\psi_{i,j})_t:\Fmc_n^+\to\Fmc_n^+$ defined by $(\psi_{i,j})_t:(N,N')\mapsto\big(N_{i,j}(t),N_{i,j}'(t)\big)$.
\item The \emph{bulging flow on $\Fmc_n^+$ associated to $a_{i,j}$} is the flow $(\beta_{i,j})_t:\Fmc_n^+\to\Fmc_n^+$ defined by $(\beta_{i,j})_t:(N,N')\mapsto\big(M_{i,j}(t),M_{i,j}'(t)\big)$.
\end{enumerate}
\end{definition}

Next, consider $i,j,k=1,\dots,n$ so that $i<j<k$. Let $p_1:=p'_i$, $p_2:=p'_j$, $p_3:=p'_3$, and note that the three edges $a_{i,j}$, $a_{j,k}$ and $a_{k,i}$ cut $N'$ into four polygons, one of which is a triangle $\Delta$ whose vertices are $p_1$, $p_2$, $p_3$. Let $M_1'$, $M_2'$, $M_3'$ be the other three polygons so that $M_1'\cup M_2'\cup M_3'\cup\Delta=N'$, enumerated so that $M_i'$ has $a_{i-1,i+1}$ as an edge. Also, let $u_1,u_2,u_3$ be as defined in Notation \ref{triple notation}. For $i=1,2,3$, let $T_i$ be the triangle in $\Delta$ whose vertices are $p_{i-1}$, $u_{i-1}$ and $p_{i+1}$ (See Figure \ref{notation picture}), and let $N_i':=M_i'\cup T_i$. Note that if $T$ is the triangle in $\Delta$ with vertices $u_1,u_2,u_3$, then $T\cup N_1'\cup N_2'\cup N_3'=N'$.

Similarly, the three edges $a_{i,j}$, $a_{j,k}$ and $a_{k,i}$ cut $N$ into four polygons (three of which might possibly be empty). As before, one of these polygons is $\Delta$. Let $M_1,M_2,M_3$ be the other three polygons so that $M_1\cup M_2\cup M_3\cup\Delta=N$, enumerated so that $M_i$ has $a_{i-1,i+1}$ as an edge. Then for $i=1,2,3$, let $N_i:=M_i\cup T_i$, and note that $T\cup N_1\cup N_2\cup N_3=N$ as well (see Figure \ref{NiMi}).

\begin{figure}[ht]
\centering
\includegraphics[scale=0.8]{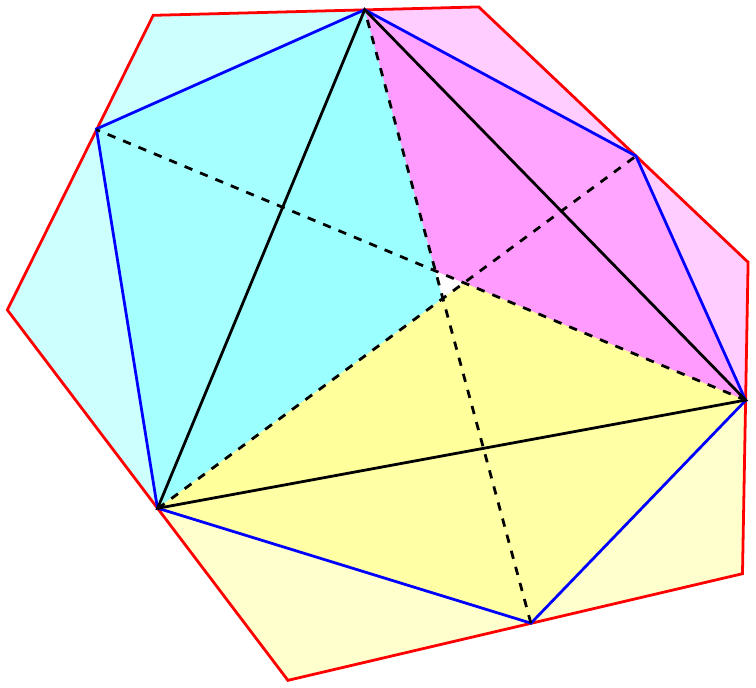}
\small
\put (-147, 37){$p_1$}
\put (-93, 160){$p_2$}
\put (0, 65){$p_3$}
\put (-60, 100){$T_1$}
\put (-72, 70){$T_2$}
\put (-100, 102){$T_3$}
\caption{$N$ is outlined in blue, $N'$ is outlined in red, $\Delta$ is outlined in black.  $N_1$, $N_1'$ are shaded in violet, $N_2$, $N_2'$ are shaded in yellow, $N_3$, $N_3'$ are shaded in turquoise, and $T$ is the white triangle.}\label{NiMi}
\end{figure}
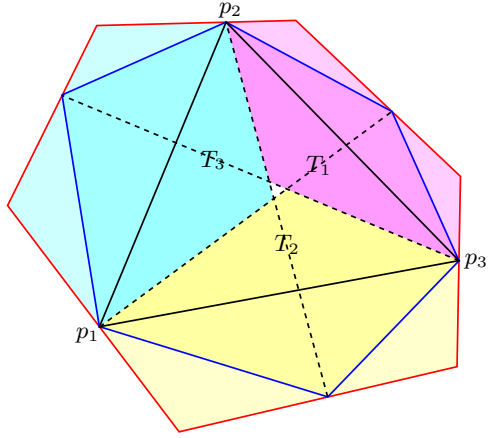

Let $g_1(t), g_2(t), g_3(t)\in\PGL(3,\Rbbb)$ be the three group elements defined in Section \ref{deforming a triple of flags} (with $p_1:=p'_i$, $p_2:=p'_j$, $p_3:=p'_3$). Then define
\begin{eqnarray*}
N_{i,j,k}(t)&:=&\big(g_1(t)\cdot N_1\big)\cup\big(g_2(t)\cdot N_2\big)\cup\big(g_3(t)\cdot N_3\big)\cup T(t),\\
N_{i,j,k}'(t)&:=&\big(g_1(t)\cdot N_1'\big)\cup\big(g_2(t)\cdot N_2'\big)\cup\big(g_3(t)\cdot N_3\big)\cup T(t).
\end{eqnarray*}
As before, note that $\big(N_{i,j,k}(t),N'_{i,j,k}(t)\big)$ is a pair of suitably nested, labelled polygons. 

\begin{definition}
The \emph{eruption flow on $\Fmc_n^+$ associated to $p_i,p_j,p_k$} is the flow $(\epsilon_{i,j,k})_t:\Fmc_n^+\to\Fmc_n^+$ defined by $(\epsilon_{i,j,k})_t:(N,N')\mapsto\big(N_{i,j,k}(t),N_{i,j,k}'(t)\big)$.
\end{definition}

Let $(N,N')=\big((\overline{p}_1,\overline{l}_1),\dots,(\overline{p}_n,\overline{l}_n)\big)\in\Fmc_n^+$ and let $\Tmc$ be a triangulation of $N$ so that the vertices of the triangulation is $\{\overline{p}_1,\dots,\overline{p}_n\}$. Let $\Imc_\Tmc$ denote the set of internal edges of $\Tmc$ and let $\Theta_\Tmc$ denote the set of triangles of $\Tmc$. The eruption flows, shearing flows and bulging flows, associated to the triangulation of $N$ gives us a family of flows in $\Fmc_n^+$, which descend to flows on $\PGL(3,\Rbbb)\backslash\Fmc_n^+$. The descended flows have the following properties. 

\begin{prop}
 Consider the collection of flows on $\PGL(3,\Rbbb)\backslash\Fmc_n^+$
\[\Mmc(\Fmc_n^+):=\{\psi_{i,j}:a_{i+1}\in\Imc_\Tmc\}\cup\{\beta_{i,j}:a_{i+1}\in\Imc_\Tmc\}\cup\{\epsilon_{i,j,k}:\{a_{i,j},a_{j,k},a_{k,i}\}\in\Theta_\Tmc\}.\]
\begin{enumerate}
\item For any $\phi_1,\phi_2\in\Mmc(\Fmc_n^+)$ and any $t_1,t_2\in\Rbbb$, $(\phi_1)_{t_1}\circ(\phi_2)_{t_2}=(\phi_2)_{t_2}\circ(\phi_1)_{t_1}$ as flows on $\PGL(3,\Rbbb)\backslash\Fmc_n^+$.
\item For any pair $F_1,F_2\in\PGL(3,\Rbbb)\backslash\Fmc_n^+$, there is a sequence $\phi_1,\dots,\phi_l\in\Mmc(\Fmc_n^+)$ and a sequence $t_1,\dots,t_l\in\Rbbb$ so that $F_1=(\phi_1)_{t_1}\circ\dots\circ(\phi_l)_{t_l}(F_2)$.
\end{enumerate}
\end{prop}

The above proposition is a consequence of the analogous statements of the propositions in Section \ref{deformations of properly convex domains} for suitably nested, labelled polygons.

\section{Deformations of properly convex domains with $C^1$ boundary} \label{deformations of properly convex domains}
In this section we define elementary eruption, shear and bulge deformations for strictly convex domains with $C^1$ boundary. Let $\Dmc$ denote the space of marked strictly convex domains in $\Rbbb\Pbbb^2$ with $C^1$ boundary, i.e.
\[\Dmc:=\left\{(\xi,\Omega):\begin{array}{l}\Omega\subset\Rbbb\Pbbb^2\text{ is a strictly convex domain with }C^1\text{ boundary}\\
\xi:S^1\to\partial{\Omega}\text{ is a homeomorphism}\end{array}\right\}.\]

The set $\Dmc$ can be topologized so that the sequence $\{(\xi_i,\Omega_i)\}_{i=1}^\infty$ converges to $(\xi,\Omega)$ if and only if $\{\Omega_i\}_{i=1}^\infty$ converges to $\Omega$ in the topology generated by the Hausdorff distance, and $\{\xi_i\}_{i=1}^\infty$ converges to $\xi$ pointwise. Each point in $\Dmc$ can be approximated by a sequence $\{(N_n,N_n')\}_{n=3}^\infty$, where each $(N_n,N_n')\in\Fmc_n^+$. The goal of this section is to define elementary shearing, bulging and eruption flows on $\Dmc$.

We fix some notation. 
\begin{notation} \label{interval notation}
\begin{itemize}
\item For any triple of pairwise distinct points $x,y,z\in S^1$ in that order, let $[x,y]_z$ and $(x,y)_z$ be the closed and open subintervals of $S^1$ with endpoints $x$ and $y$ that does not contain $z$. 
\item For any properly convex domain $\Omega\subset\Rbbb\Pbbb^2$ and any $a,b\in\overline{\Omega}$, let $[a,b]$ and $(a,b)$ denote the closed and open oriented projective line segments in $\Omega$ with $a$ and $b$ as backward and forward endpoints respectively.
\item For any $p,q\in\Rbbb\Pbbb^2$, let $\overline{pq}$ denote the projective line through $p$ and $q$.
\item For any $(\xi,\Omega)\in\Dmc$ and for any $x\in S^1$, let $\xi^*(x)$ be the tangent line to $\partial\Omega$ at $\xi(x)$.
\end{itemize}
\end{notation}

\subsection{Shearing and bulging flows}
Let $x,y\in S^1$ be a pair of distinct points. Then for all $(\xi,\Omega)\in\Dmc$, $[\xi(x),\xi(y)]$ cuts $\Omega$ into two properly convex subdomains $\Omega_{x,y,R}$ and $\Omega_{x,y,L}$, where $\Omega_{x,y,L}$ is the subdomain of $\Omega$ on the left of $[\xi(x),\xi(y)]$, and $\Omega_{x,y,R}$ is the one on the right of $[\xi(x),\xi(y)]$ (see Figure \ref{C1domain2}). Here, the orientation on $\Omega$ is induced by the homeomorphism $\xi$. 

Also, let $p_{x,y}$ be the point of intersection between $\xi^*(x)$ and $\xi^*(y)$, and let $v_x, v_y, v_{x,y}\in\Rbbb^3$ be non-zero vectors that span $\xi(x)$, $\xi(y)$ and $p_{x,y}$ respectively. As before, let $s_{x,y}(t),b_{x,y}(t)\in\PGL(3,\Rbbb)$ be the projective transformations that are represented by the matrices 
\[s_{x,y}(t):=\left[\begin{array}{ccc}
e^\frac{t}{2}&0&0\\
0&1&0\\
0&0&e^{-\frac{t}{2}}
\end{array}\right],\,\,
b_{x,y}(t):=\left[\begin{array}{ccc}
e^{-\frac{t}{6}}&0&0\\
0&e^{\frac{t}{3}}&0\\
0&0&e^{-\frac{t}{6}}
\end{array}\right],\]
respectively in the basis $\{v_x,v_{x,y},v_y\}$. 

Note that $s_{x,y}(t_1)b_{x,y}(t_2)=s_{x,y}(t_1)b_{x,y}(t_2)$ for every $t_1,t_2\in\Rbbb$. Also, it is easy to see that every projective transformation that fixes $\xi(x)$, $\xi(y)$ and $p_{x,y}$ can be written as uniquely as $s_{x,y}(t_1)b_{x,y}(t_2)$ for some $t_1,t_2\in\Rbbb$. For any such projective transformation $g$, define
\[\Omega_g:=\big(\xi(x),\xi(y)\big)\cup \left(g\cdot\Omega_{x,y,L}\right)\cup \left(g^{-1}\cdot\Omega_{x,y,R}\right).\]
Observe that $\Omega_g$ is a strictly convex domain with $C^1$ boundary. 

Let $B_R$ and $B_L$ be the two connected components of $S^1\setminus\{x,y\}$ so that for any $(\xi,\Omega)\in\Dmc$, $\xi(B_R)$ and $\xi(B_L)$ are the subsegments in the boundary of $\Omega_{x,y,R}$ and $\Omega_{x,y,L}$ respectively. Then define $\xi_g:S^1\to\partial\Omega_g$ by
\[\xi_g(a)=\left\{\begin{array}{ll}
g\circ \xi(a)&\text{if }a\in B_L\\
g^{-1}\circ \xi(a)&\text{if }a\in B_R\\
\xi(a)&\text{if }a=x,y
\end{array}\right.\]
Clearly, $\xi_g$ is continuous, so we can define the shearing flows and bulging flows on $\Dmc$ in the following way.

\begin{definition}
Let $x,y\in S^1$. 
\begin{enumerate}
\item The \emph{elementary shearing flow on $\Dmc$ associated to $(x,y)$} is the flow $(\psi_{x,y})_t:\Dmc\to\Dmc$ defined by  $(\psi_{x,y})_t:(\xi,\Omega)\mapsto(\xi_{s_{x,y}(t)},\Omega_{s_{x,y}(t)})$.
\item The \emph{elementary bulging flow on $\Dmc$ associated to $(x,y)$} is the flow $(\beta_{x,y})_t:\Dmc\to\Dmc$ defined by  $(\beta_{x,y})_t:(\xi,\Omega)\mapsto(\xi_{b_{x,y}(t)},\Omega_{b_{x,y}(t)})$.
\end{enumerate}
\end{definition}

Note that for all $g\in\PGL(3,\Rbbb)$, for all $x,y\in S^1$, and for all $t\in\Rbbb$, \[g\circ(\psi_{x,y})_t(\xi,\Omega)=(\psi_{x,y})_t\circ g(\xi,\Omega)\,\,\,\text{ and }\,\,\,g\circ(\beta_{x,y})_t(\xi,\Omega)=(\beta_{x,y})_t\circ g(\xi,\Omega).\]
This implies that the shearing and bulging flows descend to flows on $\PGL(3,\Rbbb)\backslash\Dmc$. We also denote the descended shearing and bulging flows by $\psi_{x,y}$ and $\beta_{x,y}$ respectively; it should be clear from context which we are referring to.

\begin{figure}[ht]
\centering
\includegraphics[scale=0.5]{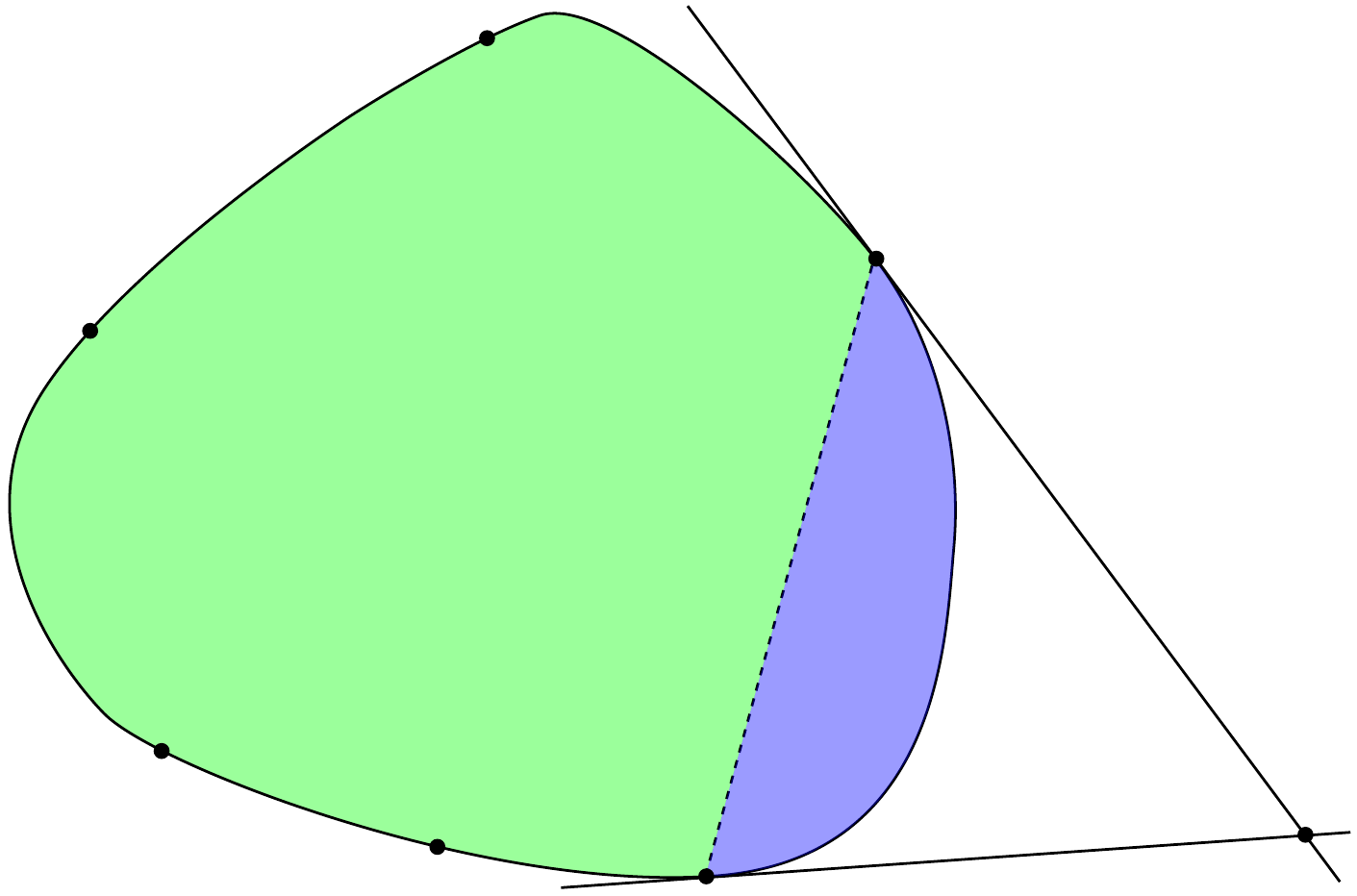}
\small
\put (-102, -5){$\xi(x_2)$}
\put (-72, 99){$\xi(x_1)$}
\put (-159, 0){$\xi(y_1)$}
\put (-200, 14){$\xi(z_1)$}
\put (-211, 88){$\xi(y_2)$}
\put (-147, 134){$\xi(z_2)$}
\put (-8, 13){$p_{x_1,x_2}$}
\caption{$\Omega_{x_1,x_2,L}$ and $\Omega_{x_1,x_2,R}$ are the regions shaded in blue and green respectively.}\label{C1domain2}
\end{figure}
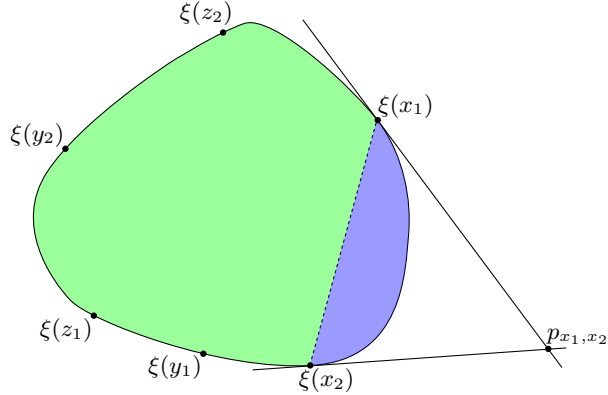

\begin{prop}\label{cross ratio facts}
Let $x_1<x_2\leq y_1<z_1<y_2<z_2\leq x_1$ be points along $S^1$ in this cyclic order (see Figure \ref{C1domain2}). Also, let $(\xi,\Omega)\in\Dmc$, and let $[\xi,\Omega]\in\PGL(3,\Rbbb)\backslash\Dmc$ be the equivalence class containing $(\xi,\Omega)$. 
\begin{enumerate}
\item For all $\phi_1,\phi_2\in\{\psi_{x_1,x_2},\beta_{x_1,x_2},\psi_{y_1,y_2},\beta_{y_1,y_2}\}$ and for all $t_1,t_2\in\Rbbb$, we have that 
\[(\phi_1)_{t_1}\circ(\phi_2)_{t_2}[\xi,\Omega]=(\phi_2)_{t_2}\circ(\phi_1)_{t_1}[\xi,\Omega].\]
\item Let $t\in\Rbbb$, and let $(\xi_1,\Omega_1):=(\psi_{x_1,x_2})_t(\xi,\Omega)$ or $(\beta_{x_1,x_2})_t(\xi,\Omega)$. Then
\[C\big(\xi^*(y_1),\xi(z_1),\xi(z_2),\overline{\xi(y_1)\xi(y_2)}\big)=C\big(\xi_1^*(y_1),\xi_1(z_1),\xi_1(z_2),\overline{\xi_1(y_1)\xi_1(y_2)}\big).\]
\end{enumerate}
\end{prop}

\begin{proof}
First, we will prove (1). Let $(\xi_1,\Omega_1):=(\phi_2)_{t_2}(\xi,\Omega)$, $(\xi_2,\Omega_2):=(\phi_1)_{t_1}(\xi_1,\Omega_1)$, $(\xi_3,\Omega_3):=(\phi_1)_{t_1}(\xi,\Omega)$, and $(\xi_4,\Omega_4):=(\phi_2)_{t_2}(\xi_3,\Omega_3)$. We need to show that there is some projective transformation $g\in\PGL(3,\Rbbb)$ so that $g\circ\xi_4=\xi_2$. It is sufficient to do so on the four subsegments $[x_1,x_2]_{y_1}$, $[x_2,y_1]_{y_2}$, $[y_1,y_2]_{x_1}$ and $[y_2,x_1]_{x_2}$. 

Let $g_1,g_2\in\PGL(3,\Rbbb)$ be the projective transformations so that 
\begin{eqnarray*}
\xi_1|_{[x_1,x_2]_{y_1}}=g_1\circ\xi|_{[x_1,x_2]_{y_1}},&& \xi_1|_{\partial\Gamma\setminus[x_1,x_2]_{y_1}}=g_1^{-1}\circ\xi|_{\partial\Gamma\setminus[x_1,x_2]_{y_1}},\\
\xi_3|_{[y_1,y_2]_{x_1}}=g_2\circ\xi|_{[y_1,y_2]_{x_1}},&&\xi_3|_{\partial\Gamma\setminus[y_1,y_2]_{x_1}}=g_2^{-1}\circ\xi|_{\partial\Gamma\setminus[y_1,y_2]_{x_1}}.
\end{eqnarray*}
Observe then that
\begin{eqnarray*}
\xi_2|_{[y_1,y_2]_{x_1}}&=&g_1^{-1}\circ g_2\circ g_1\circ\xi_1|_{[y_1,y_2]_{x_1}},\\
\xi_2|_{\partial\Gamma\setminus[y_1,y_2]_{x_1}}&=&g_1^{-1}\circ g_2^{-1}\circ g_1\circ\xi_1|_{\partial\Gamma\setminus[y_1,y_2]_{x_1}},\\
\xi_4|_{[x_1,x_2]_{y_1}}&=&g_2^{-1}\circ g_1\circ g_2\circ\xi_3|_{[x_1,x_2]_{y_1}},\\
\xi_4|_{\partial\Gamma\setminus[x_1,x_2]_{y_1}}&=&g_2^{-1}\circ g_1^{-1}\circ g_2\circ\xi_3|_{\partial\Gamma\setminus[x_1,x_2]_{y_1}},
\end{eqnarray*}
which implies that
\begin{eqnarray*}
\xi_2|_{[x_1,x_2]_{y_1}}=g_1^{-1}\circ g_2^{-1}\circ g_1^2\circ\xi|_{[x_1,x_2]_{y_1}},&&\xi_2|_{[x_2,y_1]_{y_2}}=g_1^{-1}\circ g_2^{-1}\circ\xi|_{[x_2,y_1]_{y_2}},\\
\xi_2|_{[y_1,y_2]_{x_1}}=g_1^{-1}\circ g_2\circ\xi|_{[y_1,y_2]_{x_1}},&&
\xi_2|_{[y_2,x_1]_{x_2}}=g_1^{-1}\circ g_2^{-1}\circ\xi|_{[y_2,x_1]_{x_2}},\\
\xi_4|_{[x_1,x_2]_{y_1}}=g_2^{-1}\circ g_1\circ\xi|_{[x_1,x_2]_{y_1}},&&
\xi_4|_{[x_2,y_1]_{y_2}}=g_2^{-1}\circ g_1^{-1}\circ\xi|_{[x_2,y_1]_{y_2}},\\
\xi_4|_{[y_1,y_2]_{x_1}}=g_2^{-1}\circ g_1^{-1}\circ g_2^2\circ\xi|_{[y_1,y_2]_{x_1}},&&
\xi_4|_{[y_2,x_1]_{x_2}}=g_2^{-1}\circ g_1^{-1}\circ\xi|_{[y_2,x_1]_{x_2}}.
\end{eqnarray*}
In all cases, $g_1^{-1}\circ g_2^{-1}\circ g_1\circ g_2\circ\xi_4=\xi_2$.

To see (2), one only needs to observe that the flows $(\psi_{x_1,x_2})_t$ and $(\beta_{x_1,x_2})_t$ change the flags $\big(\xi(y_1),\xi^*(y_1)\big)$, $\big(\xi(y_2),\xi^*(y_2)\big)$, $\big(\xi(z_1),\xi^*(z_1)\big)$, $\big(\xi(z_2),\xi^*(z_2)\big)$ by the same projective transformation. 
\end{proof}

For any $(\xi,\Omega)\in\Dmc$ and any $y_1<z_1<y_2<z_2<y_1$ along $S^1$ in this cyclic order, we have $ C\big(\xi^*(y_1),\xi(z_1),\xi(z_2),\overline{\xi(y_1)\xi(y_2)}\big)<0$ and can thus define
\[\sigma_\xi(y_1,z_1,z_2,y_2):=\log\Big(-C\big(\xi^*(y_1),\xi(z_1),\xi(z_2),\overline{\xi(y_1)\xi(y_2)}\big)\Big).\]

The next proposition states how certain cross ratios change when we perform the shearing and bulging flows. The proof is a straightforward calculation, which we omit.

\begin{prop} Let $y_1\neq y_2$ be points along $S^1$, and for $i=L,R$, let $z_i\in B_i$ with respect to the oriented line segment $[y_1,y_2]$. Also, let $t\in\Rbbb$, let $(\xi_1,\Omega_1):=(\psi_{y_1,y_2})_t(\xi,\Omega)$, and let $(\xi_2,\Omega_2):=(\beta_{y_1,y_2})_t(\xi,\Omega)$. Then
\begin{enumerate}
\item $\sigma_{\xi_1}(y_1,z_L,z_R,y_2)=\sigma_{\xi}(y_1,z_L,z_R,y_2)-t,$
\item $\sigma_{\xi_1}(y_2,z_R,z_L,y_1)=\sigma_{\xi}(y_2,z_R,z_L,y_1)-t,$
\item $\sigma_{\xi_2}(y_1,z_L,z_R,y_2)=\sigma_{\xi}(y_1,z_L,z_R,y_2)+t,$
\item $\sigma_{\xi_2}(y_2,z_R,z_L,y_1)=\sigma_{\xi}(y_2,z_R,z_L,y_1)-t,$
\end{enumerate}
\end{prop}

\subsection{Eruption flows}
Fix any $x_1,x_2,x_3\in S^1$ that are pairwise distinct. For any $(\xi,\Omega)\in\Dmc$ and for $i=1,2,3$, let $p_i:=\xi(x_i)$ and let $l_i:=\xi^*(x_i)$. Then let $u_1,u_2,u_3$ be as defined in Notation \ref{triple notation}, and for all $i=1,2,3$, let $\Omega_i$ be the subdomain in $\Omega$ bounded by $[p_{i-1},p_{i+1}]_{p_i}$, $[p_{i-1},u_i]$ and $[u_i,p_{i+1}]$ (recall that arithmetic in the subscripts are done modulo 3). Observe that there is a unique triangle $T$ so that
\[\Omega=\Omega_1\cup\Omega_2\cup\Omega_3\cup T.\]
More concretely, $T$ is the triangle in $\Omega$ with vertices $u_1,u_2,u_3$ (see Figure \ref{C1domain}).

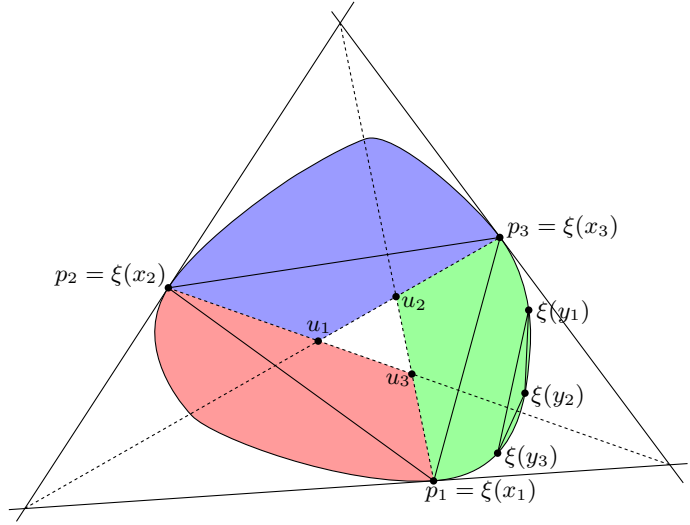
\begin{figure}[ht]
\centering
\includegraphics[scale=0.5]{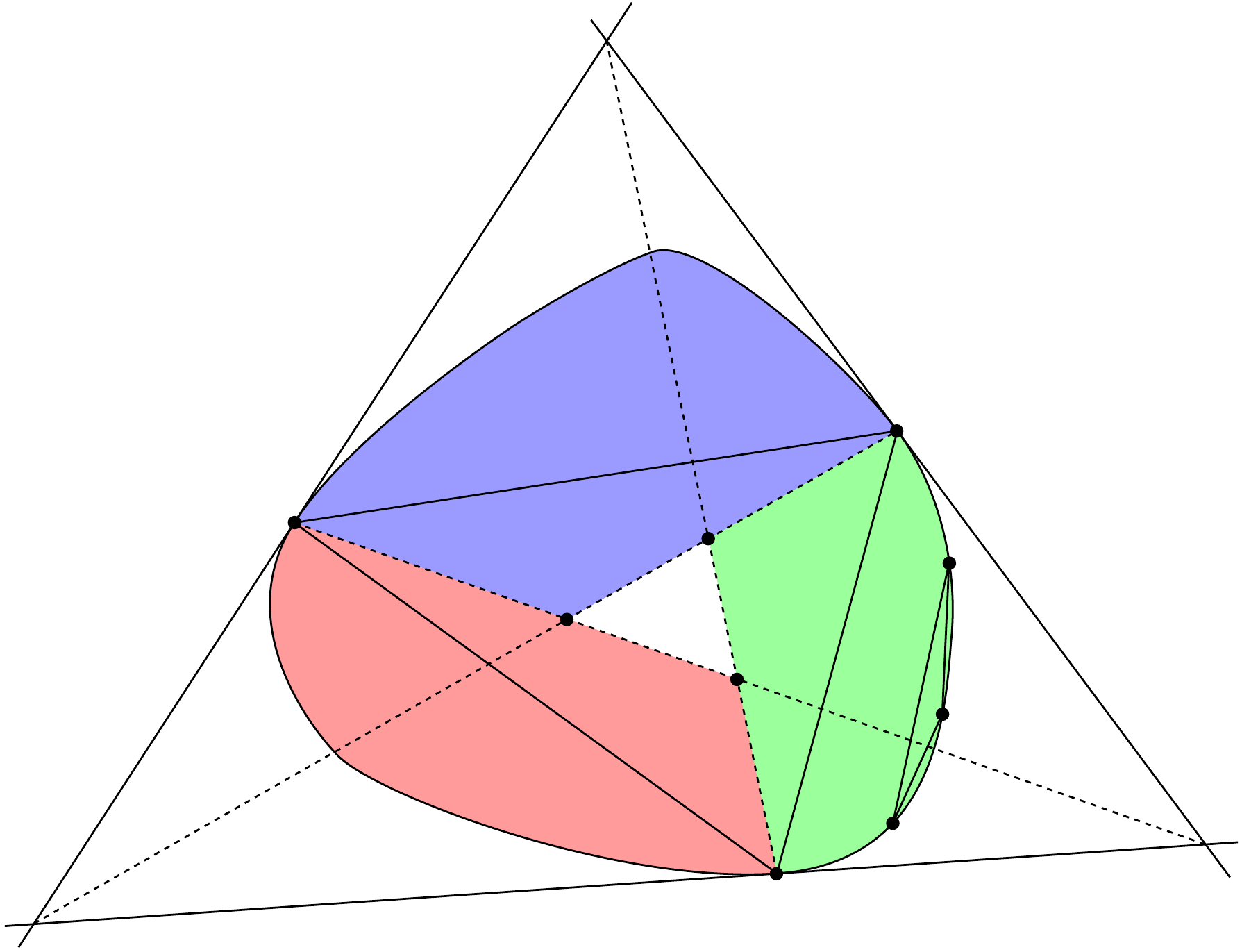}
\small
\put (-100, 10){$p_1=\xi(x_1)$}
\put (-240, 92){$p_2=\xi(x_2)$}
\put (-69, 110){$p_3=\xi(x_3)$}
\put (-58, 78){$\xi(y_1)$}
\put (-60, 46){$\xi(y_2)$}
\put (-70, 23){$\xi(y_3)$}
\put (-145, 73){$u_1$}
\put (-109, 82){$u_2$}
\put (-115, 53){$u_3$}
\caption{$\Omega_1$, $\Omega_2$ and $\Omega_3$ are the regions shaded in blue, green and red respectively.}\label{C1domain}
\end{figure}

Let $g_1(t),g_2(t),g_3(t)\in\PGL(3,\Rbbb)$ be group elements as defined in Section \ref{deforming a triple of flags}. It is an easy computation to check that for all $i=1,2,3$, $g_{i+1}(t)\cdot p_i=g_{i-1}(t)\cdot p_i$, and $g_i(t)\cdot u_i$ lies on $g_i(t)\cdot w_{i-1}$. These together imply that there is a unique triangle $T(t)$ so that 
\[\Omega_{x_1,x_2,x_3,t}:=\big(g_1(t)\cdot\Omega_1\big)\cup \big(g_2(t)\cdot\Omega_2\big)\cup \big(g_3(t)\cdot\Omega_3\big)\cup T(t)\]
is a strictly convex domain with $C^1$ boundary. Furthermore, for $i=1,2,3$, the tangent lines to $\partial\Omega_{x_1,x_2,x_3,t}$ at $g_i(s)\cdot \xi(x_i)$ is $l_i$.

Also, let $\xi_{x_1,x_2,x_3,t}:S^1\to\partial\Omega_{p_1.p_2.p_3,t}$ be the map defined by
\[\xi_{x_1,x_2,x_3,t}(a)=\left\{\begin{array}{ll}
g_1(t)\circ \xi(a)&\text{if }a\in [x_2,x_3]_{x_1}\\
g_2(t)\circ \xi(a)&\text{if }a\in [x_3,x_1]_{x_2}\\
g_3(t)\circ \xi(a)&\text{if }a\in [x_1,x_2]_{x_3}\\
\end{array}\right.,\]
and note that $\xi_{x_1,x_2,x_3,t}$ is well-defined and continuous. 

\begin{definition}
Let $x_1,x_2,x_3\in S^1$ be triple of pairwise distinct points. The \emph{elementary eruption flow on $\Dmc$ associated to $x_1,x_2,x_3$} is the flow $(\epsilon_{x_1,x_2,x_3})_t:\Dmc\to\Dmc$ defined by $(\epsilon_{x_1,x_2,x_3})_t:(\xi,\Omega)\mapsto(\xi_{x_1,x_2,x_3,t},\Omega_{x_1,x_2,x_3,t})$
\end{definition}

Like the shearing and bulging flows, we also have that for all $g\in\PGL(3,\Rbbb)$ and all $x_1,x_2,x_3\in S^1$ that are pairwise distinct, $g\circ(\epsilon_{x_1,x_2,x_3})_t=(\epsilon_{x_1,x_2,x_3})_t\circ g$. Thus, the eruption flow also descends to a flow on $\PGL(3,\Rbbb)\backslash\Dmc$, which we also denote by $(\epsilon_{x_1,x_2,x_3})_t:\PGL(3,\Rbbb)\backslash\Dmc\to\PGL(3,\Rbbb)\backslash\Dmc$ and refer to as the \emph{eruption flow} on $\PGL(3,\Rbbb)\backslash\Dmc$. 

The next two propositions state how the triple ratios change under the eruption flow. The proofs are a variation of the proof of Proposition \ref{cross ratio facts}, so we will leave them to the reader.

\begin{prop}\label{triple ratio facts}
Let $x_1<x_2<x_3\leq y_1<y_2<y_3\leq x_1$ lie along $S^1$ in this cyclic order, let $t_1,t_2\in\Rbbb$, and let $[\xi,\Omega]\in\PGL(3,\Rbbb)\backslash\Dmc$ be the equivalence class containing $(\xi,\Omega)\in\Dmc$ (see Figure \ref{C1domain}). The following statements hold:
\begin{enumerate}
\item $\displaystyle(\epsilon_{x_1,x_2,x_3})_{t_1}\circ(\epsilon_{y_1,y_2,y_3})_{t_2}[\xi,\Omega]=(\epsilon_{y_1,y_2,y_3})_{t_2}\circ(\epsilon_{x_1,x_2,x_3})_{t_1}[\xi,\Omega]$.
\item For any $\phi\in\{\psi_{x_1,x_3},\beta_{x_1,x_3}\}$, we have that 
\[\phi_{t_1}\circ(\epsilon_{y_1,y_2,y_3})_{t_2}[\xi,\Omega]=(\epsilon_{y_1,y_2,y_3})_{t_2}\circ\phi_{t_1}[\xi,\Omega].\]
\item Let $t\in\Rbbb$ and let $(\xi_1,\Omega_1):=(\epsilon_{x_1,x_2,x_3})_t(\xi,\Omega).$ Then 
\begin{eqnarray*}
&&T\Big(\big(\xi^*(y_1),\xi(y_1)\big),\big(\xi^*(y_2),\xi(y_2)\big),\big(\xi^*(y_3),\xi(y_3)\big)\Big)\\ 
&=&T\Big(\big(\xi_1^*(y_1),\xi_1(y_1)\big),\big(\xi_1^*(y_2),\xi_1(y_2)\big),\big(\xi_1^*(y_3),\xi_1(y_3)\big)\Big).
\end{eqnarray*}
\end{enumerate}
\end{prop}

For any $(\xi,\Omega)\in\Dmc$ and any $x_1,x_2,x_3\in S^1$ that are pairwise distinct, we have
\[T\Big(\big(\xi^*(y_1),\xi(y_1)\big),\big(\xi^*(y_2),\xi(y_2)\big),\big(\xi^*(y_3),\xi(y_3)\big)\Big)>0,\] 
so we can define
\[\tau_\xi(x_1,x_2,x_3):=\log T\Big(\big(\xi^*(x_1),\xi(x_1)\big),\big(\xi^*(x_2),\xi(x_2)\big),\big(\xi^*(x_3),\xi(x_3)\big)\Big).\]

\begin{prop}\label{triple compute}
Let $x_1,x_2,x_3\in S^1$ and let $(\xi,\Omega)\in\Dmc$. For all $t\in\Rbbb$, let $(\epsilon_{x_1,x_2,x_3})_t(\xi,\Omega)=(\xi_1,\Omega_1)$. Then
\[\tau_{\xi_1}(x_1,x_2,x_3)=\tau_\xi(x_1,x_2,x_3)+t.\]
\end{prop}

Finally, the next proposition says that when we perform the shearing, bulging, or eruption flows, we are only changing the pair $(\xi,\Omega)$ ``near the edges" involved in defining these flows. The proof uses the same argument as the proof of (2) of Proposition \ref{cross ratio facts}, so we will omit it.

\begin{prop}\label{triple ratio and cross ratio facts}
Let $x_1<x_2<x_3\leq y_1<y_2<y_3\leq x_1$ lie along $S^1$ in this cyclic order. If $y_1=x_3$ and $y_3=x_1$, let $y_4:=x_2$. Otherwise, let $y_4\in(x_3,y_1)_{y_2}\cup(y_3,x_1)_{y_2}$ (see Notation~\ref{interval notation}). Also, let $t\in\Rbbb$ and let $(\xi,\Omega)\in\Dmc$. The following statements hold:
\begin{enumerate}
\item Let $(\xi_1,\Omega_1):=(\epsilon_{x_1,x_2,x_3})_t(\xi,\Omega).$ Then 
\begin{eqnarray*}
\sigma_{\xi,\xi^*}(y_1,y_2,y_4,y_3)&=&\sigma_{\xi_1,\xi_1^*}(y_1,y_2,y_4,y_3),\\
\sigma_{\xi,\xi^*}(y_3,y_4,y_2,y_1)&=&\sigma_{\xi_1,\xi_1^*}(y_3,y_4,y_2,y_1).\\
\end{eqnarray*}
\item Let $\phi\in\{\psi_{x_1,x_3},\beta_{x_1,x_3}\}$ and let $(\xi_1,\Omega_1):=\phi_t(\xi,\Omega).$ Then 
\[\tau_{\xi,\xi^*}(y_1,y_2,y_3)=\tau_{\xi_1,\xi_1^*}(y_1,y_2,y_3).\]
\end{enumerate}
\end{prop}

\section{Convex real projective structures on $S$} 

Now we consider $S$, a closed, connected, orientable smooth surface of genus $ g\geq 2$. In Section~\ref{deform_convex} we extend the elementary shearing, bulging, and eruption flows to flows on the space of convex real projective structures on $S$. In this section we describe the necessary background in the deformation space of convex real projective structures on $S$.

\subsection{Basics}\label{basics} 
We recall the definition of convex real projective structures on $S$, and briefly describe some of their properties.

 A \emph{convex $\Rbbb\Pbbb^2$ surface} $\Sigma$ is the quotient of a properly convex domain $\Omega$ in $\Rbbb\Pbbb^2$ by a group of projective transformations $\Gamma$ that acts freely, properly discontinuously and cocompactly on $\Omega$, i.e. $\Sigma=\Gamma\backslash\Omega$. If $\Sigma$ and $\Sigma'$ are convex $\Rbbb\Pbbb^2$ surfaces, a \emph{(projective) isomorphism} is a diffeomorphism $f:\Sigma\to\Sigma'$,  whose induced map on the universal covers $\widetilde{f}:\widetilde{\Sigma}\to\widetilde{\Sigma}'$ is the restriction of a projective transformation on $\Rbbb\Pbbb^2$.
We consider 
\[\widetilde{\Cmc(S)}:=\left\{(f,\Sigma):\begin{array}{l}
\Sigma\text{ is a convex }\Rbbb\Pbbb^2\text{ surface }\\
f:S\to\Sigma\text{ is a diffeomorphism}\end{array}\right\},\]
and define the \emph{deformation space of convex $\Rbbb\Pbbb^2$ structures on $S$} to be 
\[\Cmc(S):=\widetilde{\Cmc(S)}/\sim,\]
where $(f,\Sigma)\sim(f',\Sigma')$ if $f'\circ f^{-1}:\Sigma\to\Sigma'$ is isotopic to a projective isomorphism from $\Sigma$ to $\Sigma'$. An element $[f,\Sigma]\in\Cmc(S)$ is a \emph{convex $\Rbbb\Pbbb^2$ structure} on $S$.
 For any $(f,\Gamma\backslash\Omega)\in\widetilde{\Cmc(S)}$, the diffeomorphism $f$ lifts to a map $\widetilde{f}:\widetilde{S}\to\Omega$, which is called the \emph{developing map}. Since $\widetilde{f}$ intertwines the deck transformations on $\widetilde{S}$ and $\Omega$, there is a unique group homomorphism $\text{hol}:\pi_1(S)\to\Gamma\subset\PGL(3,\Rbbb)$, called the \emph{holonomy representation}, so that $f$ is $\text{hol}$-equivariant.

Consider the space 
\[\widetilde{\Hmc(S)}:=\left\{(F,\text{hol}):\begin{array}{ll}\text{hol}:\pi_1(S)\to\PGL(3,\Rbbb)\text{ is a homomorphism},\\ F:\widetilde{S}\to\Rbbb\Pbbb^2\text{ is a }\text{hol}\text{-equivariant diffeomorphism} \\ \text{onto a properly convex domain }\Omega\subset\Rbbb\Pbbb^2\end{array}\right\}.\]
The map $\widetilde{\Phi}:\widetilde{\Cmc(S)}\to\widetilde{\Hmc(S)}$ given by $\widetilde{\Phi}:(f,\Gamma\backslash\Omega)\mapsto(\widetilde{f},\text{hol})$ is a bijection, where $\widetilde{f}$ and $\text{hol}$ are the developing map and the holonomy representation for $(f,\Gamma\backslash\Omega)$ respectively. The topology on $\widetilde{\Hmc(S)}$ can be described as follows: a sequence $\{(F_i,\rho_i)\}_{i=1}^\infty$ converges to $(F,\rho)$ in $\widetilde{\Hmc(S)}$ if and only if $\{F_i\}_{i=1}^\infty$ converges to $F$ pointwise. It is easy to see that this implies that $\{\rho_i\}_{i=1}^\infty$ converges to $\rho$  in $\text{Hom}(\pi_1(S),\PGL(3,\Rbbb))$. This then induces a topology $\widetilde{\Cmc(S)}$ via $\widetilde{\Phi}$.

Let $\widehat{\Hmc(S)}:=\widetilde{\Hmc(S)}/\sim$, where $(F,\rho)\sim(F',\rho')$ if $\rho=\rho'$ and $F$ is isotopic to $F'$ via $\rho$-equivariant diffeomorphisms onto properly convex subsets of $\Rbbb\Pbbb^2$. Also, let $\Hmc(S):=\widehat{\Hmc(S)}/\sim$, where $h_1\sim h_2$ if there are representatives $(F_i,\rho_i)$ of $h_i$ and $g\in\PGL(3,\Rbbb)$ so that $(g\circ F_1,g\rho_1(\cdot)g^{-1})=(F_2,\rho_2)$. One can verify that $\widetilde{\Phi}$ descends to a homeomorphism $\Phi:\Cmc(S)\to\Hmc(S)$. 
 
Let $\widetilde{\text{Hol}}:\widetilde{\Hmc(S)}\to\text{Hom}(\pi_1(S),\PGL(3,\Rbbb))$ be defined by $\widetilde{\text{Hol}}:(F,\text{hol})\mapsto\text{hol}$. Choi-Goldman \cite{ChoiGoldman} proved that this map descends to a map $\text{Hol}:\Hmc(S)\to\text{Hom}(\pi_1(S),\PGL(3,\Rbbb))/\PGL(3,\Rbbb)$, which is a homeomorphism onto a connected component of $\text{Hom}(\pi_1(S),\PGL(3,\Rbbb))/\PGL(3,\Rbbb)$. We similarly denote by $\widetilde{\text{Hol}}$ and $\text{Hol} $ the maps induces by precomposition with $\widetilde{\Phi}$ and $\Phi$: 
\begin{eqnarray*}
\widetilde{\text{Hol}}:\widetilde{\Cmc(S)}\to\text{Hom}(\pi_1(S),\PGL(3,\Rbbb)).\\
\text{Hol}:\Cmc(S)\to\text{Hom}(\pi_1(S),\PGL(3,\Rbbb))/\PGL(3,\Rbbb).
\end{eqnarray*}
Furthermore, it is known that the image of $\text{hol}$ does not contain any singular points of $\text{Hom}(\pi_1(S),\PGL(3,\Rbbb))/\PGL(3,\Rbbb)$, so we can use the map $\text{hol}$ to define a smooth structure on $\Cmc(S)$ and $\Hmc(S)$.

The following theorem lists some classical result of Kuiper \cite{Kuiper} and Benzecri \cite{Benzecri}, (also see Theorem 3.2 of \cite{Goldman_convex}). 

\begin{thm} \label{properties of convex projective surfaces} Let $[f,\Gamma\backslash\Omega]\in\Cmc(S)$, let $[\rho]=\mathrm{Hol}[f,\Gamma\backslash\Omega]$, and let $\gamma\in\pi_1(S)\setminus\{\id\}$. 
\begin{enumerate}
\item The image of $f$ is a strictly convex domain in $\Rbbb\Pbbb^2$ with $C^1$ boundary. 
\item The group element $\rho(\gamma)\in\PGL(3,\Rbbb)$ is diagonalizable over $\Rbbb$ with eigenvalues having pairwise distinct absolute values. Furthermore, the unique representative of $\rho(\gamma)$ in $\SL(3,\Rbbb)$ has only positive eigenvalues.
\item Let $\rho(\gamma)^+$, $\rho(\gamma)^0$ and $\rho(\gamma)^-$ be the attracting, neutral and repelling fixed points of $\rho(\gamma)$ in $\Rbbb\Pbbb^2$. Then $\rho(\gamma)^+$ and $\rho(\gamma)^-$ lie on $\partial\Omega$, and the projective lines tangent to $\partial\Omega$ at $\rho(\gamma)^+$ and $\rho(\gamma)^-$ both contain $\rho(\gamma)^0$.\end{enumerate}
\end{thm}

It is well-known that $\pi_1(S)$ is a hyperbolic group. For each $(F,\rho)\in\widetilde{\Hmc(S)}$, $\rho$ induces a $\pi_1(S)$ action on $\Omega:=F(\widetilde{S})$ by projective transformations. In particular, $\pi_1(S)$ acts on $\Omega$ by isometries of the Hilbert metric, which induces a homeomorphism $\xi:\partial\pi_1(S)\to\partial\Omega$, where $\partial\pi_1(S)$ is the Gromov boundary of $\pi_1(S)$. Note that the map $\xi$ depends only on the equivalence class $[F,\rho]\in \widehat{\Hmc(S)}$ of $(F,\rho)$. Part (1) of Theorem \ref{properties of convex projective surfaces} then allows us to define the embedding $\widehat{\Hmc(S)}\to\Dmc$ given by $[F,\rho]\mapsto(\xi,\Omega)$, which descends to an embedding $\Hmc(S)\to\PGL(3,\Rbbb)\backslash\Dmc$. We use this embedding in Section~\ref{deform_convex} to define the shearing flows, bulging flows and eruption flows to $\Cmc(S)$.

\subsection{An ideal triangulation of $S$} \label{ideal triangulation}
We describe now a particular ideal triangulation on $S$ that will be used later. The triangulation depends on some topological choices on $S$. We choose an orientation on $S$ and a pants decomposition of $S$. Let $P_1,\dots,P_{2g-2}$ be the pairs of pants given by the choice of pants decomposition. For each $i=1,\dots,2g-2$, the inclusion of $P_i\subset S$ induces an embedding $\pi_1(P_i)\subset\pi_1(S)$ (after suitable choices of base points). 

For each $P_i$, let $\gamma_{1,i}, \gamma_{2,i}, \gamma_{3,i}\in\pi_1(P_i)\subset\pi_1(S)$ be three group elements corresponding to oriented peripheral curves in $P_i$ so that $\gamma_{3,i}\cdot\gamma_{2,i}\cdot\gamma_{1,i}=\id$, and $P_i$ lies on the left of each of the oriented boundary components of $P_i$ corresponding to the conjugacy classes $[\gamma_{1,i}]$, $[\gamma_{2,i}]$ and $[\gamma_{3,i}]$.

Let $\Gmc(\widetilde{S}):=\{(x,y)\in\partial\pi_1(S)^2:x\neq y\}/\big((x,y)\sim(y,x)\big)$, and denote every equivalence class in $\Gmc(\widetilde{S})$ by $\{x,y\}$. If we choose a negatively curved metric on $S$, then $\Gmc(\widetilde{S})$ is naturally identified with the space of geodesics on $\widetilde{S}$ in this negatively curved metric. As such, we will refer to the elements in $\Gmc(\widetilde{S})$ as \emph{geodesics of $\widetilde{S}$}. Also, we say that a pair of geodesics $\{x,y\},\{z,w\}\in\Gmc(\widetilde{S})$ \emph{intersects transversely} if $x,z,y,w\in\partial\pi_1(S)$ in this cyclic order.

Let $\gamma_{1,i}^-, \gamma_{2,i}^-, \gamma_{3,i}^-\in\partial\pi_1(S)$ and $\gamma_{1,i}^+, \gamma_{2,i}^+, \gamma_{3,i}^+\in\partial\pi_1(S)$ be the repelling and attracting fixed points of $\gamma_{1,i}$, $\gamma_{2,i}$, $\gamma_{3,i}$ respectively. Then let
\[\widetilde{\Qmc}_i:=\bigcup_{\gamma\in\pi_1(S)}\big\{\gamma\cdot\{\gamma_{1,i}^-,\gamma_{2,i}^-\},\gamma\cdot\{\gamma_{2,i}^-,\gamma_{3,i}^-\},\gamma\cdot\{\gamma_{3,i}^-,\gamma_{1,i}^-\}\big\},\]
\[\widetilde{\Pmc}:=\bigcup_{i=1}^{2g-2}\bigcup_{\gamma\in\pi_1(S)}\big\{\gamma\cdot\{\gamma_{1,i}^-,\gamma_{1,i}^+\},\gamma\cdot\{\gamma_{2,i}^-,\gamma_{2,i}^+\},\gamma\cdot\{\gamma_{3,i}^-,\gamma_{3,i}^+\}\big\}.\]
Clearly, $\widetilde{\Pmc}$ and $\widetilde{\Qmc}_i$ for all $i=1,\dots,2g-2$ are invariant under the $\pi_1(S)$ action on $\partial\pi_1(S)^2$. Thus, we can define
\[\Qmc_i:=\pi_1(S)\backslash\widetilde{\Qmc}_i,\,\,\,\,\widetilde{\Qmc}:=\bigcup_{i=1}^{2g-2}\widetilde{\Qmc}_i,\,\,\,\,\widetilde{\Tmc}:=\widetilde{\Qmc}\cup\widetilde{\Pmc},\]
\[\Pmc:=\pi_1(S)\backslash\widetilde{\Pmc},\,\,\,\,\Qmc:=\pi_1(S)\backslash\widetilde{\Qmc},\,\,\,\,\Tmc:=\pi_1(S)\backslash\widetilde{\Tmc}\]
and observe that $\Qmc=\bigcup_{i=1}^{2g-2}\Qmc_i$ and $\Tmc=\Qmc\cup\Pmc$ (see Figure \ref{triangulation}). 

\begin{figure}[ht]
\centering
\includegraphics[scale=0.65]{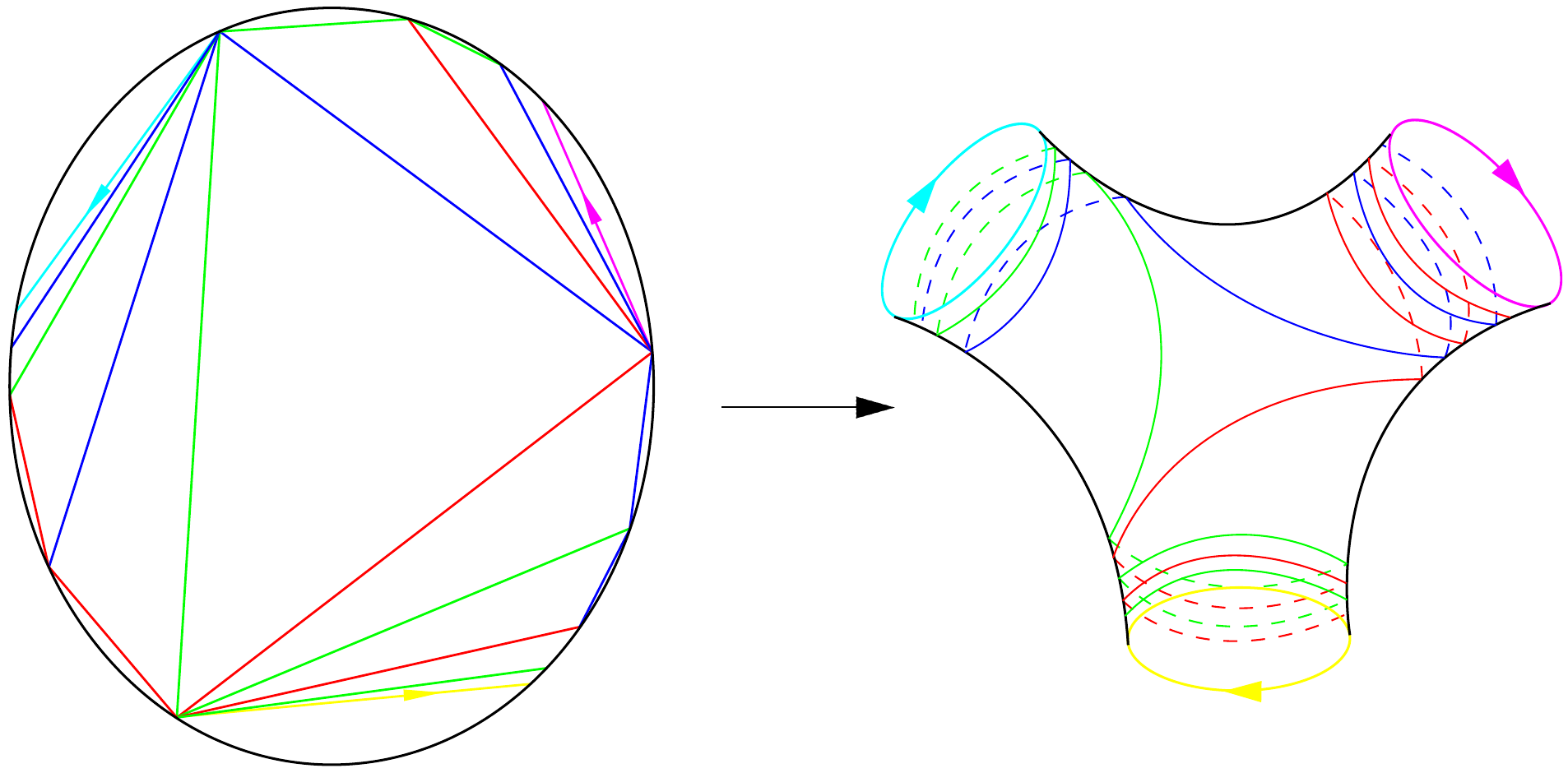}
\small
\put (-320, 174){$\gamma_{1,i}^-$}
\put (-369, 108){$\gamma_{1,i}^+$}
\put (-330, 7){$\gamma_{2,i}^-$}
\put (-238, 15){$\gamma_{2,i}^+$}
\put (-208, 97){$\gamma_{3,i}^-$}
\put (-235, 155){$\gamma_{3,i}^+$}
\caption{The colored curves are edges in $\widetilde{\Tmc}$ and $\Tmc$. The red, blue and green curves are in $\widetilde{\Qmc}$ and $\Qmc$, while the violet, turquoise and yellow curves are in $\widetilde{\Pmc}$ and $\Pmc$}\label{triangulation}
\end{figure}

The geodesics in $\widetilde{\Qmc}$ can be characterized as the geodesics $\{x,y\}\in\widetilde{\Tmc}$ for which there exists $z,z'\in\partial\pi_1(S)$ so that $\{x,z\},\{y,z\},\{x,z'\},\{y,z'\}\in\widetilde{\Tmc}$. Similarly, the geodesics in $\widetilde{\Pmc}$ are exactly the geodesics $\{x,y\}\in\widetilde{\Tmc}$ for which there exists sequences $\{z_i\}_{i=1}^\infty,\{z'_i\}_{i=1}^\infty\subset\partial\pi_1(S)$ so that $\lim_{i\to\infty}z_i'=x$, $\lim_{i\to\infty}z_i=y$, and $\{x,z_i\},\{y,z'_i\}\in\widetilde{\Tmc}$ for all $i=1,\dots,\infty$. It is also easy to see that $\Pmc$ defined above is naturally in bijection with the pants decomposition we chose on $S$, and that $|\Qmc_i|=3$ for all $i=1,\dots,2g-2$. In particular, $|\Pmc|=3g-3$ and $|\Qmc|=6g-6$. 

If we choose a negatively curved metric on $S$, then $\Tmc$ can be naturally realized as an ideal triangulation of $S$, where each pair of pants given by $\Pmc$ is cut into two ideal triangles. As such, we will call $\Tmc$ (resp. $\widetilde{\Tmc}$) an \emph{ideal triangulation} of $S$ (resp. $\widetilde{S}$), and the elements in $\Tmc$ and $\widetilde{\Tmc}$ are called \emph{edges}. The edges of $\widetilde{\Tmc}$ that lie in $\widetilde{\Pmc}$ (resp. $\widetilde{\Qmc}$) are called \emph{closed edges} (resp. \emph{non-closed edges}). Also, a \emph{triangle} of $\widetilde{\Tmc}$ is a triple of edges of the form $T=\big\{\{x,y\},\{y,z\},\{z,x\}\big\}\subset\widetilde{\Qmc}$, and each of $\{x,y\}$, $\{y,z\}$, $\{z,x\}$ is an \emph{edge} of $T$. 

Similarly, the closed edges, non-closed edges, and triangles of $\Tmc$ are images of closed edges, non-closed edges, and triangles of $\widetilde{\Tmc}$ under the quotient map $\widetilde{\Tmc}\to\Tmc$. Denote the set of triangles of $\widetilde{\Tmc}$ and $\Tmc$ by $\Theta_{\widetilde{\Tmc}}$ and $\Theta_\Tmc$ respectively.  Note that for all $i=1,\dots,2g-2$, there are exactly two triangles in $\Theta_\Tmc$ whose edges are the three edges in $\Qmc_i$. As such, $|\Theta_\Tmc|=4g-4$.

\subsection{Bonahon-Dreyer parameterization of $\Cmc(S)$}\label{BonahonDreyer}
The first parametrization of the deformation space of convex real projective structures $\mathcal{C}(S)$ was given by Goldman \cite{Goldman_convex}. Bonahon-Dreyer \cite{BonahonDreyer1} gave a parameterization of the $\PSL(n,\Rbbb)$-Hitchin component associated to $S$, which coincides with $\mathcal{C}(S)$ if $n=3$ for any given ideal triangulation of $S$, building upon coordinate systems introduced by Fock-Goncharov \cite{FockGoncharov} for decorated local systems on surfaces with boundary. 
The second author \cite{Zhang_internal} later gave a reparametrization of the Bonahon-Dreyer coordinates, which gives Fenchel-Nielsen type coordinates on the $\PSL(n,\Rbbb)$-Hitchin component. We describe the Bonahon-Dreyer parameterization in the special case of $\Cmc(S)$ for our particular choice of ideal triangulation $\Tmc$ of $S$. For the explicit change of coordinates between Goldman's coordinates and the Bonahon-Dreyer coordinates see \cite{BonahonKim}. 

For each closed edge $e\in\Pmc$, choose a representative $\{x,y\}\in\widetilde{\Pmc}$ of $e$, and choose a pair of vertices $z=z_{x,y},z'=z'_{x,y}\in\partial\pi_1(S)$ so that $\{z,x\},\{z',y\}\in\widetilde{\Qmc}$. For any non-closed edge $e\in\Qmc$, choose a representative $\{x,y\}\in\widetilde{\Qmc}$ for $e$, and let $z=z_{x,y},z'=z'_{x,y}\in\partial\pi_1(S)$ be the points so that $\{x,z\},\{y,z\},\{x,z'\},\{y,z'\}\in\widetilde{\Tmc}$. In either case, $x<z<y<z'<x$ lie along $\partial\pi_1(S)$ in this cyclic order. With this, we can define, for any $e\in\Tmc$, the functions $\sigma_{e,x},\sigma_{e,y}:\widetilde{\Cmc(S)}=\widetilde{\Hmc(S)}\to\Rbbb$ by 
\[\sigma_{e,x}(\xi,\Omega):=\sigma_{\xi,\xi^*}(x,z,z',y)\,\,\,\text{ and }\,\,\,\sigma_{e,y}(\xi,\Omega):=\sigma_{\xi,\xi^*}(y,z',z,x).\]

Since $\sigma_{e,x}$ and $\sigma_{e,y}$ are defined using projective invariants, they do not depend on the choice of representative for $e$, and so are indeed well-defined. For the same reason, they descend to functions, also denoted $\sigma_{e,x}$ and $\sigma_{e,y}$, on $\Cmc(S)=\Hmc(S)$. These descended functions are called the \emph{shear parameters} on $\Cmc(S)$. Since $|\Tmc|=9g-9$, we have $2(9g-9)$ shear parameters on $\Cmc(S)$.

For each triangle $T=\big\{\{x,y\},\{y,z\},\{z,x\}\big\}$ of $\widetilde{\Tmc}$ so that $x<y<z<x$ in this cyclic order along $\partial\pi_1(S)$, the triple ratio gives rise to a function $\tau_T:\widetilde{\Cmc(S)}=\widetilde{\Hmc(S)}\to\Rbbb$ by $\tau_T(\xi,\Omega):=\tau_\xi(x,y,z)$. As before, this function is defined using a projective invariant, so it descends to a function, also denoted $\tau_T$, on $\Cmc(S)=\Hmc(S)$. There are $4g-4$ such functions, and they are called the \emph{triangle parameters} on $\Cmc(S)$.

The shear and triangle parameters satisfy certain linear relations which we now describe. 
For each closed edge $e\in\Pmc$, let $\{x,y\}$ be a representative of $e$, and let $z_1,z_2,z'_1,z'_2\in\partial\pi_1(S)$ so that $\{z_1,x\}$, $\{z_2,x\}$, $\{z_1,z_2\}$, $\{z_1',y\}$, $\{z_2',y\}$, $\{z_1',z_2'\}\in\widetilde{\Tmc}$. For $i=1,2$, let $e_i\in\Tmc$ be the equivalence class containing $\{z_i,x\}$ and let $e_3$ be the equivalence class containing $\{z_1,z_2\}$. Similarly, for $i=1,2$, let $e_i'\in\Tmc$ be the equivalence class containing $\{z_i',y\}$ and let $e_3'$ be the equivalence class containing $\{z_1',z_2'\}$. Then, let $T_1$, $T_2$ be the two triangles whose edges are $e_1,e_2,e_3$ and let $T_1'$, $T_2'$ be the triangles whose edges are $e_1',e_2',e_3'$. 

One can then explicitly compute that for each closed edge $e\in\Pmc$, the following two identities and inequalities hold on $\Cmc(S)$:
\[\sigma_{e_1,x}+\sigma_{e_2,x}=\sigma_{e_1',x}+\sigma_{e_2',x}+\tau_{T_1'}+\tau_{T_2'}>0\]
and
\[\sigma_{e_1,y}+\sigma_{e_2,y}+\tau_{T_1}+\tau_{T_2}=\sigma_{e_1',y}+\sigma_{e_2',y}>0.\]
The two equalities for each $e\in\Pmc$ are called the \emph{closed leaf equalities}, and the two inequalities for each $e\in\Pmc$ are called the \emph{closed leaf inequalities}. The two quantities $\sigma_{e,x}$ and $\sigma_{e,y}$ also have the following geometric interpretation. For any $[f,\Sigma]\in\Cmc(S)$, let $[\rho]=\textrm{Hol}[f,\Sigma]$. Also, let $\gamma$ be the primitive element in $\pi_1(S)$ with $x$ and $y$ as its repelling and attracting fixed points respectively, and let $|\lambda_1(\rho)|>|\lambda_2(\rho)|>|\lambda_3(\rho)|$ denote the absolute values of the eigenvalues of $\rho(\gamma)$. Then Bonahon-Dreyer computed that the two quantities above are $\log|\frac{\lambda_1}{\lambda_2}|$ and $\log|\frac{\lambda_2}{\lambda_3}|$ respectively.

The following theorem of Bonahon-Dreyer states that the closed leaf equalities and inequalities are the only relations between the shear and triangle parameters.

\begin{thm}\cite[Theorem17]{BonahonDreyer1}\label{Bonahon-Dreyer}
The shear and triangle parameters give a real analytic diffeomorphism from $\Cmc(S)$ to a convex polytope $P$ in $\Rbbb^{22g-22}$ of dimension $16g-16$ that is cut out by the closed leaf equalities and inequalities described above.
\end{thm}

\subsection{Families of tranverse curves} 
A convex real projective structure on $S$ is identified by its holonomy representation $\text{hol}: \pi_1(S) \to \PGL(3,\Rbbb)$ and its developing map $\widetilde{f}:\widetilde{S}\to\Omega$, which is $\text{hol}$-equivariant, see Section~\ref{basics}. This gives an identification $\xi:\partial\pi_1(S) = S^1\rightarrow \partial \Omega$ and defines a point $(\xi,\Omega)\in\Dmc$. The ideal triangulation $\Tmc$ of $S$ induces an ideal triangulation $\widetilde{\Tmc}$ of $\Omega$. 
In this section we consider, for any pair of distinct vertices $x_0,y_0$ of $\widetilde{\Tmc}$, the set of curves $\Emc_{x_0,y_0}$ which intersect the projective segment with endpoint $x_0, y_0$ in $\Omega$ transversely. 
The goal is to get a decomposition of $\Emc_{x_0,y_0}$ into a finite family of curves, and a special infinite family of curves, which allows us to give a controlled analysis of how shearing and triangle parameters change under the eruption and internal bulging flows we introduce in Section~\ref{deform_convex}. 

We let $\partial\pi_1(S)$ be the circle $S^1$ in the definition of $\Dmc$. Then the ideal triangulation $\Tmc$ of $S$ induces an ideal triangulation on each $(\xi,\Omega)\in\Dmc$. For any pair of distinct vertices $x_0,y_0$ of $\widetilde{\Tmc}$, let 
\[\Emc'_{x_0,y_0}:=\big\{\{x,y\}\in\widetilde{\Tmc}:x<x_0<y<y_0<x\big\}.\]
(Note that if $\{x_0,y_0\}\in\widetilde{\Tmc}$, then $\Emc'_{x_0,y_0}$ is empty.) If we choose $(\xi,\Omega)\in\Dmc$ and a complete, negatively curved metric on $\Omega$, then $\Emc'_{x_0,y_0}$ is the set of edges in $\widetilde{\Tmc}$ that intersect the geodesic in $\Omega$ with endpoints $x_0$, $y_0$ transversely. By orienting both components of $\partial\pi_1(S)\setminus\{x_0,y_0\}$ from $x_0$ to $y_0$, this induces an ordering on $\Emc'_{x_0,y_0}$ by $\{x,y\}<\{x',y'\}$ if $x$ and $x'$ (hence $y$ and $y'$) lie in the same connected component of $\partial\pi_1(S)\setminus\{x_0,y_0\}$, $x$ weakly precedes $x'$, and $y$ weakly precedes $y'$. 

Observe that $\Emc'_{x_0,y_0}$ does not have a minimum (in the ordering described above) if and only if there is some vertex $z$ of $\widetilde{\Tmc}$ so that $\{x_0,z\}\in\widetilde{\Pmc}$, and there is a sequence $\{x_i\}_{i=1}^\infty$ of vertices of $\widetilde{\Tmc}$ that converges to $x_0$, and $\{x_i,z\}\in\Emc'_{x_0,y_0}$ for all $i$. Similarly, $\Emc'_{x_0,y_0}$ does not have a maximum if and only if there is some vertex $z'$ of $\widetilde{\Tmc}$ so that $\{y_0,z'\}\in\widetilde{\Pmc}$, and there is a sequence $\{y_i\}_{i=1}^\infty$ of vertices of $\widetilde{\Tmc}$ that converges to $y_0$, and $\{y_i,z\}\in\Emc'_{x_0,y_0}$ for all $i$. Then define
\[\Emc_{x_0,y_0}:=\left\{\begin{array}{ll}\Emc'_{x_0,y_0}&\text{if }\Emc'_{x_0,y_0}\text{ has a max and a min}\\
\Emc'_{x_0,y_0}\cup\big\{\{x_0,z\}\big\}&\text{if }\Emc'_{x_0,y_0}\text{ has a max but no min}\\
\Emc'_{x_0,y_0}\cup\big\{\{y_0,z'\}\big\}&\text{if }\Emc'_{x_0,y_0}\text{ has a min but no max}\\
\Emc'_{x_0,y_0}\cup\big\{\{x_0,z\},\{y_0,z'\}\big\}&\text{if }\Emc'_{x_0,y_0}\text{ has neither a max nor a min}
\end{array}\right.\]
and observe that $\Emc_{x_0,y_0}$ has an obvious ordering (see Figure \ref{Exy}).

\begin{figure}[ht]
\centering
\includegraphics[scale=0.8]{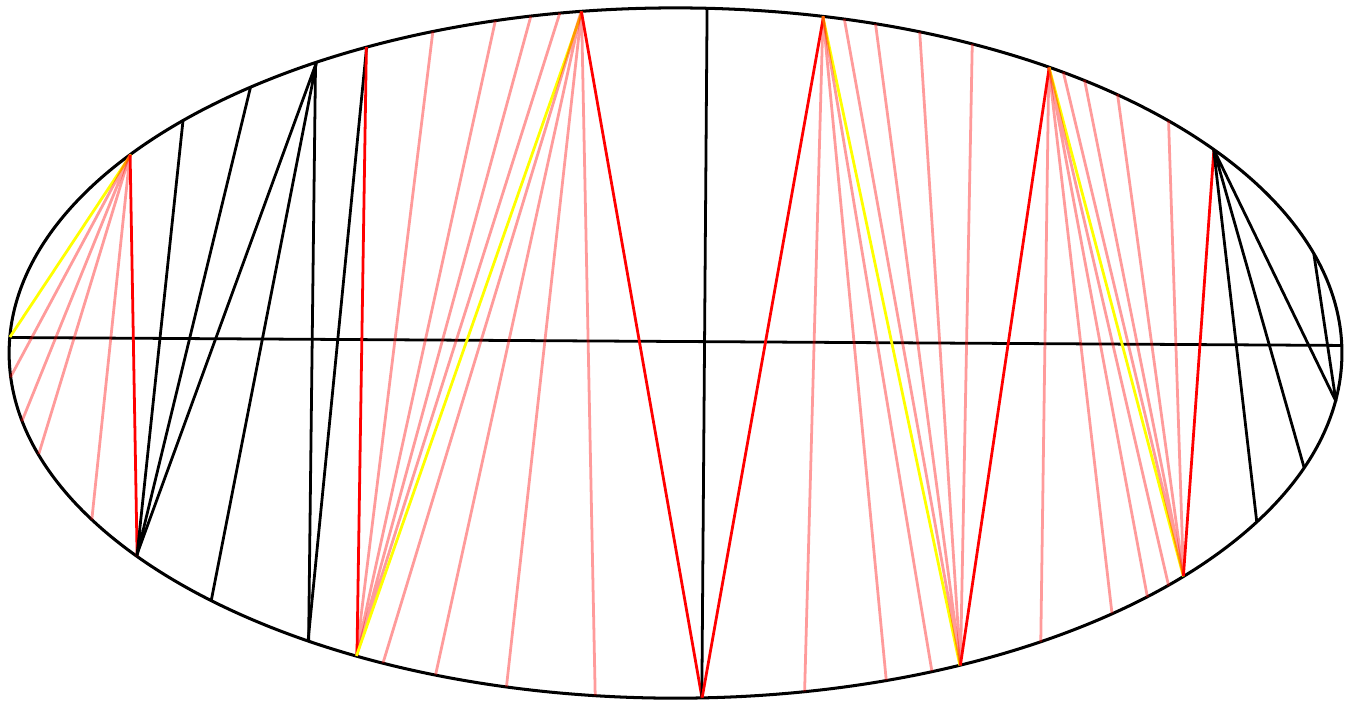}
\small
\put (-320, 84){$x_0$}
\put (0, 83){$y_0$}
\caption{The geodesics drawn above that intersect $\{x_0,y_0\}$ transversely represent the edges in $\Emc_{x_0,y_0}$. The yellow geodesics are edges in $\widetilde{\Pmc}$, the colored geodesics are edges in $\Emc_j$ and the black geodesics are edges in $\Emc_{j,j+1}$ for some $j$. Also, the non-faded geodesics are edges in $\overline{\Emc}_{x_0,y_0}$}\label{Exy}
\end{figure}

It is easy to see that there are only finitely many (possibly none) edges in $\widetilde{\Pmc}$ that lie in $\Emc_{x_0,y_0}$. Let $l_1,\dots,l_k$ denote these edges, enumerated according to the ordering on $\Emc_{x_0,y_0}$. Observe that if $e\in\Emc_{x_0,y_0}$ shares a common vertex $x$ with some $l_j$ and satisfies $e<l_j$, then every edge $e'$ satisfying $e<e'<l_j$ also has $x$ as a vertex. Similarly, if $e\in\Emc_{x_0,y_0}$ shares a common vertex $x$ with some $l_j$ and $l_j<e$, then every edge $e'$ satisfying $l_j<e'<e$ also has $x$ as a vertex. Thus, if we define
\[\Emc_j:=\{e\in\Emc_{x_0,y_0}:e\text{ shares a vertex with }l_j\},\]
\[\Fmc_j^-:=\{e\in\Emc_{x_0,y_0}:e<e'\text{ for all }e'\in\Emc_j\},\]
\[\Fmc_j^+:=\{e\in\Emc_{x_0,y_0}:e>e'\text{ for all }e'\in\Emc_j\},\]
then $\Emc_{x_0,y_0}=\Fmc_j^-\cup\Emc_j\cup\Fmc_j^+$ is a disjoint union. 
Note that $\Emc_j$ is infinite for all $j$. 

We further define 
\[\Emc_{j,j+1}:=\left\{\begin{array}{ll}
\Fmc_1^-&\text{if }j=0\\
\Fmc_j^+\cap\Fmc_{j+1}^-&\text{if }0<j<k\\
\Fmc_n^+&\text{if }j=k\\
\end{array}\right.\]
and note that $\displaystyle\Emc_{x_0,y_0}=\bigcup_{j=1}^k\Emc_j\cup\bigcup_{j=0}^k\Emc_{j,j+1}$ is a disjoint union. Note also that $\Emc_{j,j+1}$ is finite for all $j$. 

 Observe that each $\Emc_j$ has a minimum, a maximum, and a unique edge in $\widetilde{\Pmc}$. The edge in $\Pmc$ might possibly be the minimum if $j=1$, and might possibly be the maximum if $j=k$. On the other hand, if $j=2,\dots,k-1$, then the edge in $\Pmc$ is neither the minimum nor the maximum. In any case, replace each $\Emc_j\subset\Emc_{x_0,y_0}$ with these three edges, which we denote by $a_j\leq b_j\leq c_j$, to get a finite ordered set $\overline{\Emc}_{x_0,y_0}$. 

It is clear that if $a_j\neq b_j$, then $a_j$ shares a vertex with the maximum of $\Emc_{j-1,j}$ (or with $c_{j-1}$ if $\Emc_{j-1,j}$ is empty). Similarly, if $b_j\neq c_j$, then $c_j$ shares a vertex with the minimum of $\Emc_{j,j+1}$ (or with $a_{j+1}$ if $\Emc_{j,j+1}$ is empty). In particular, every pair of adjacent edges in $\overline{\Emc}_{x_0,y_0}$ determines a unique triple of points in $\partial\pi_1(S)$, and hence a triangle (which might not be in $\Theta_\Tmc$). Let $\Theta_{x_0,y_0}$ denote this finite collection of triangles.

\section{Deformations of convex $\Rbbb\Pbbb^2$ structures on $S$}\label{deform_convex}
We cannot simply apply the elementary shearing, bulging and eruption flows defined in Section \ref{deformations of properly convex domains} to obtain a deformation of a convex real projective structure on $S$, because these elementary flows are not equivariant with respect to the $\pi_1(S)$-action on $\Omega$. To get an equivariant flow we have to perform these flows equivariantly along an infinite family of lines and/or triangles. However then we are faced with the question of convergence, and in general we will not obtain a well-defined flow on $\Hmc(S)=\Cmc(S)$. In order to ensure well-definedness, we restrict to special  infinite combinations of elementary shearing, bulging and eruption flows. These give rise to the classical shearing and bulging flows along simple closed curves \cite{Goldman_twist, Goldman_bulging}, and to new eruption flows and internal bulging flows,  which are related to the internal triangle and shear parameters. We describe this flow without giving fully detailed proofs of all statements. In a forthcoming paper, joint with Zhe Sun, we define more general flows on the Hitchin component for $\mathrm{PSL}(n,\Rbbb)$, of which the shearing, bulging, eruption and internal bulging flows on $\Hmc(S)$ are a special case. There, complete proofs will be provided.

Let $P_i\subset S$ be any pair of pants given by the pants decomposition $\Pmc$. Let $\pi:\widetilde{S}\to S$ denote the covering map, and note that the set
\[\{x\in \widetilde{S}:\pi(x)\in P_i\}\]
has countably many connected components. Choose an enumeration of these connected components, $\widetilde{P}_{i,1},\widetilde{P}_{i,2},\dots$. Each of these $\widetilde{P}_{i,j}$ can be completely described by a subset of $\partial\pi_1(S)$, so it makes sense to say when a geodesic in $\Gmc(\widetilde{S})$ lies in $\widetilde{P}_{i,j}$.

\begin{figure}[ht]
\centering
\includegraphics[scale=0.8]{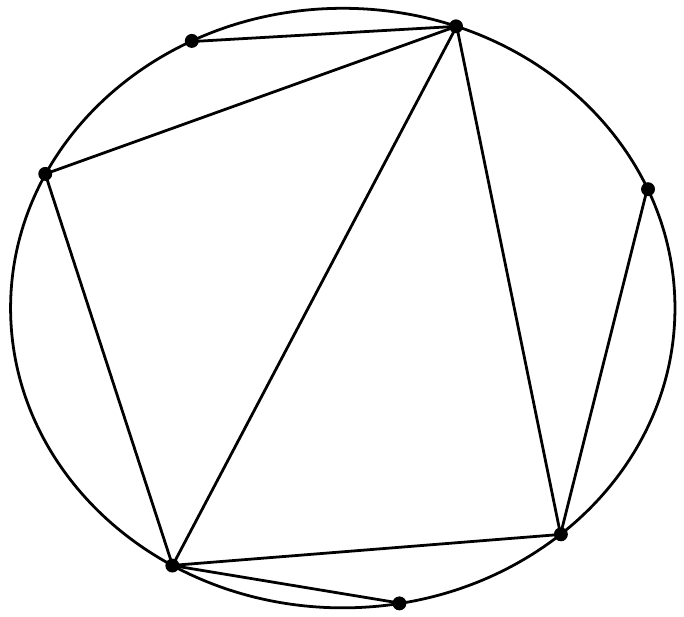}
\small
\put (-133, 7){$\gamma_{1,i}^-$}
\put (-70, -5){$\gamma_{1,i}^+$}
\put (-30, 13){$\gamma_{2,i}^-$}
\put (-7, 100){$\gamma_{2,i}^+$}
\put (-59, 142){$\gamma_{3,i}^-$}
\put (-125, 138){$\gamma_{3,i}^+$}
\put (-185, 104){$\gamma_{1,i}^{-1}\cdot\gamma_{2,i}^-$.}
\put (-70, 55){$\Delta_i$}
\put (-115, 85){$\Delta_i'$}
\caption{$\Delta_i$ and $\Delta_i'$ in $\widetilde{P}_{i,1}$}\label{standard}
\end{figure}
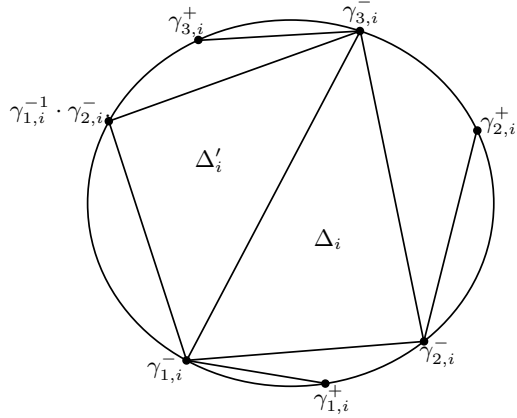

As before, choose group elements $\gamma_{1,i}, \gamma_{2,i}, \gamma_{3,i}\in\pi_1(S)$ so that 
\begin{itemize}
\item $\gamma_{3,i}\cdot\gamma_{2,i}\cdot\gamma_{1,i}=\id$, 
\item the axis of $\gamma_{j,i}$ for all $j=1,2,3$ lie in $\widetilde{P}_{i,1}$
\item $P_i$ lies on the left of each of the oriented boundary components of $P_i$ corresponding to the conjugacy classes $[\gamma_{1,i}]$, $[\gamma_{2,i}]$ and $[\gamma_{3,i}]$.
\end{itemize}
Then let  
\[\begin{array}{lll}
e_{1,i}:=\{\gamma_{1,i}^-,\gamma_{2,i}^-\}, &e_{2,i}:=\{\gamma_{2,i}^-,\gamma_{3,i}^-\},&e_{3,i}:=\{\gamma_{3,i}^-,\gamma_{1,i}^-\},\\
e'_{1,i}:=\{\gamma_{1,i}^-,\gamma_{3,i}^-\},&e'_{2,i}:=\{\gamma_{3,i}^-,\gamma_{1,i}^{-1}\cdot \gamma_{2,i}^-\},&e'_{3,i}:=\{\gamma_{1,i}^{-1}\cdot \gamma_{2,i}^-,\gamma_{1,i}^-\}
\end{array}\]
and recall that for $j=1,2,3$, $e_{j,i}$ and $e'_{j,i}$ is an edge in $\widetilde{\Tmc}$ (see Figure \ref{standard}). Define $\Delta_i:=\{e_{1,i},e_{2,i},e_{3,i}\}$, $\Delta'_i:=\{e'_{1,i},e'_{2,i},e'_{3,i}\}$, which are two triangles in $\Theta_{\widetilde{\Tmc}}$. Finally, let 
\begin{eqnarray*}
\Theta_i&:=&\{\gamma\cdot \Delta_i:\gamma\in\pi_1(S)\}\cup\{\gamma\cdot \Delta_i':\gamma\in\pi_1(S)\},\\
\Theta_{1,i}&:=&\{\gamma\cdot \Delta_i:\gamma\in\Gamma_{1,i}\}\cup\{\gamma\cdot \Delta_i':\gamma\in\Gamma_{1,i}\},\\
\Theta_{j,i}&:=&\gamma_j\cdot\Theta_{1,i},
\end{eqnarray*}
where $\Gamma_{1,i}:=\langle\gamma_{1,i},\gamma_{2,i},\gamma_{3,i}\rangle$, and $\gamma_j\in\pi_1(S)$ is a group element so that $\gamma_j\cdot\widetilde{P}_{i,1}=\widetilde{P}_{i,j}$. It is easy to see that $\Theta_i=\bigcup_{j=1}^\infty\Theta_{j,i}$.

We will now use $\Delta_i$ to iteratively define a sequence $\{N_k\}_{k=1}^\infty$ of nested collections of triangles that exhaust $\Theta_{1,i}$. First, let $N_1:=\Delta_i$. Then for all $k=2,\dots,\infty$, let $N_k$ be the collection of triangles that share a common edge with some triangle in $N_{k-1}$. It is clear that each $N_k$ is a finite collection, and $\bigcup_{k=1}^\infty N_k=\Theta_{1,i}$.

Next, choose an enumeration of $\Theta_{1,i}=\{T_{1,1},T_{2,1},\dots\}$ with the property that if $T_{l,1}\in N_k$ and $T_{l',1}\in N_{k'}\setminus N_k$ for some $k'>k$, then $l'>l$. For each $j>1$, the group element $\gamma_j\in\pi_1(S)$ induces an enumeration of $\Theta_{j,i}=\{T_{1,j},T_{2,j},\dots\}$. Using this, we can enumerate $\Theta_i$ so that the ordering on $\Theta_i$ induced by this enumeration restricts to the ordering on $\Theta_{j,i}$ induced by the enumeration of $\Theta_{j,i}$ specified above. For example, one can enumerate $\Theta_i$ by $\Theta_i=\{T_1,T_2,\dots\}$, where
$T_1:=T_{1,1}$, $T_2:=T_{2,1}$, $T_3:=T_{1,2}$, $T_4:=T_{3,1}$, $T_5:=T_{2,2}$, $T_6:=T_{1,3},\dots$. 

Similarly, let
\begin{eqnarray*}
\widetilde{\Qmc}_{1,i}&:=&\{\gamma\cdot e_{1,i}:\gamma\in\Gamma_{1,i}\}\cup\{\gamma\cdot e_{2,i}:\gamma\in\Gamma_{1,i}\}\cup\{\gamma\cdot e_{3,i}:\gamma\in\Gamma_{1,i}\},\\
\widetilde{\Qmc}_{j,i}&:=&\gamma_j\cdot\Qmc_{1,i},
\end{eqnarray*}
and as before, we have $\widetilde{\Qmc}_i=\bigcup_{j=1}^\infty\widetilde{\Qmc}_{j,i}$. Also, note that $\widetilde{\Qmc}_{j,i}$ is exactly the set of edges of the triangles in $\Theta_{j,i}$. Enumerate $\widetilde{\Qmc}_{1,i}=\{f_{1,1}:=e_{1,i},f_{2,1}:=e_{2,i},\dots\}$ so that if $f_{l,1}$ is an edge of a triangle in $N_k$ and $f_{l',1}$ is an edge of a triangle in $N_{k'}$ but not an edge of a triangle in $N_k$ for some $k'>k$, then $l'>l$. By the choice of $\gamma_j$ for each $j>1$, this enumeration of $\widetilde{\Qmc}_{1,i}$ induces an enumeration of $\widetilde{\Qmc}_{j,i}=\{f_{1,j},f_{2,j},\dots\}$. Using this, we can enumerate $\widetilde{\Qmc}_i$ by $\widetilde{\Qmc}_i=\{e_1,e_2,\dots\}$, where $e_1:=f_{1,1}$, $e_2:=f_{2,1}$, $e_3:=f_{1,2}$, $e_4:=f_{3,1}$, $e_5:=f_{2,2}$, $e_6:=f_{1,3},\dots$. 

\subsection{Shearing and bulging flows} 

First, we define the shearing and bulging flows associated to the closed leaves of $\Tmc$. These flows are special cases of the generalized twist flows previously considered by Goldman \cite{Goldman_twist}. Goldman also gave a geometric construction of these flows in \cite{Goldman_bulging}. 

Enumerate the elements in $\Pmc$ by $c_1,\dots,c_{3g-3}$. For each $i=1,\dots,3g-3$, choose an orientation on $c_i$. This induces an orientation on each edge in 
\[\Cmc_i:=\{d\in\widetilde{\Pmc}:[d]=c_i\}.\]
Furthermore, $\Cmc_i$ is a countable set, so we can enumerate $\Cmc_i=\{d_1,d_2,\dots\}$. 

For each $d_j\in\Cmc_i$, let $d_j^+$ and $d_j^-$ be the its forward and backward endpoints respectively. Then for any $t\in\Rbbb$ and any $j\in\Zbbb^+$, let 
\[(\beta_j)_t,(\psi_j)_t:\PGL(3,\Rbbb)\backslash\Dmc\to\PGL(3,\Rbbb)\backslash\Dmc\]
be the maps defined by $(\beta_j)_t:=(\beta_{d_j^+,d_j^-})_t$ and $(\psi_j)_t:=(\psi_{d_j^+,d_j^-})_t$. This allows us to define $(B_i)_t,(S_i)_t:\Cmc(S)\to\Cmc(S)$ by $(B_i)_t:=\prod_{j=1}^\infty\beta_j(t)$ and $(S_i)_t:=\prod_{j=1}^\infty\psi_j(t)$, where $\Cmc(S)$ is viewed as a subset of $\PGL(3,\Rbbb)\backslash\Dmc$ via the embedding described at the end of Section \ref{basics}. 

\begin{definition}
The flows $(B_i)_t$ and $(S_i)_t$ on $\Cmc(S)$ defined above are respectively the \emph{bulging flows} and \emph{shearing flows} corresponding to the simple closed curve $c_i\in\Pmc$.
\end{definition}

Using similar arguments as in the proofs of Theorem \ref{main theorem} and Theorem \ref{shear} below give the following theorem. In fact, since any pair of vertices in 
$\widetilde{\Tmc}$ intersects only finitely many elements in $\widetilde{\Pmc}$ the proof is this case is much simpler.  Thus we get the following theorem, which is a special case of Goldman's theorem \cite{Goldman_twist, Goldman_bulging} 
 
\begin{thm}
For all $i=1,\dots,3g-3$, $(B_i)_t$ and $(S_i)_t$ are well-defined, smooth, and do not depend on any of the choices we made to obtain the enumeration of $\widetilde{\Cmc}_i$.  
\end{thm}

\subsection{Eruption flows} 
Now we define the eruption flows. For each $T_j\in\Theta_i$, let $a_j,b_j,c_j$ be the three vertices of $T_j$. Then for any $t\in\Rbbb$, let $(\epsilon_j)_t:\PGL(3,\Rbbb)\backslash\Dmc\to\PGL(3,\Rbbb)\backslash\Dmc$ be the eruption flow defined by
\[(\epsilon_j)_t:=\left\{\begin{array}{ll}(\epsilon_{a_j,b_j,c_j})_t&\text{if }T_j=\gamma\cdot \Delta_i\\
(\epsilon_{a_j,b_j,c_j})_{-t}&\text{if }T_j=\gamma\cdot \Delta_i'
\end{array}\right..\]
Then let $(E_i)_t:=\prod_{j=1}^\infty\epsilon_j(t):\Cmc(S)\to\Cmc(S)$, where $\Cmc(S)$ is again viewed as a subset of $\PGL(3,\Rbbb)\backslash\Dmc$. 

\begin{definition}
The flow $(E_i)_t$ on $\Cmc(S)$ defined above is the \emph{eruption flow} corresponding to the pair of pants $P_i\subset S$.
\end{definition}

The main theorem of this section is the well-definedness of $(E_i)_t$.

\begin{thm} \label{main theorem}
For all $i=1,\dots,2g-2$, $(E_i)_t$ is well-defined, smooth, and does not depend on any of the choices we made to obtain the enumeration of $\Theta_i$.  
\end{thm}

\begin{remark} 
Note that for the eruption flow on $\Cmc(S)$ we simultaneously apply the elementary eruption flow in positive direction to the triangles in the $\pi_1(S)$-orbit of $\Delta_i$ and the elementary eruption flow in negative direction to the triangles in the $\pi_1(S)$-orbit of $\Delta_i'$. This is crucial in order to prove convergence and get a well-defined flow $(E_i)_t$ on $\Cmc(S)$.
 
That this is a natural and necessary thing to do can be seen from the closed leaf equalities of Bonahon-Dreyer, see Section~\ref{BonahonDreyer}: 
\[\sigma_{e_1,x}+\sigma_{e_2,x}=\sigma_{e_1',x}+\sigma_{e_2',x}+\tau_{T_1'}+\tau_{T_2'}=\log\left|\frac{\lambda_1}{\lambda_2}\right|\]
and
\[\sigma_{e_1,y}+\sigma_{e_2,y}+\tau_{T_1}+\tau_{T_2}=\sigma_{e_1',y}+\sigma_{e_2',y}=\log\left|\frac{\lambda_2}{\lambda_3}\right|.\]
Using the closed leaf inequalities for the three pants curves for $P_i$, it is easy to see that if we do not want to change the eigenvalues of the holonomy around the three pants curves, then the sum of the two triangle parameters for the two triangles in $P_i$ is necessarily constant.

The eruption flows are even more natural in Zhang's reparametrization of the Bonahon-Dreyer coordinates, where only one of the triangle invariants is taken into account - the other one being determined by the closed leaf equalities. 
\end{remark}

Choose any $[f,\Sigma]\in\Cmc(S)$, and let $(\xi,\Omega)\in\Dmc$ be a point so that $[\xi,\Omega]\in\PGL(3,\Rbbb)\backslash\Dmc$ corresponds to $[f,\Sigma]$. For all $j=1,\dots,\infty$, let 
\[(\xi_j',\Omega_j'):=\prod_{k=1}^j(\epsilon_j)_t(\xi,\Omega)\in\Dmc.\] 
Then let $g_j\in\PGL(3,\Rbbb)$ be the group element so that 
\begin{eqnarray*}
g_j\cdot \big(\xi_j'(\gamma_{1,i}^-),(\xi_j')^*(\gamma_{1,i}^-)\big)&=&\big(\xi(\gamma_{1,i}^-),\xi^*(\gamma_{1,i}^-)\big),\\
g_j\cdot \big(\xi_j'(\gamma_{2,i}^-),(\xi_j')^*(\gamma_{2,i}^-)\big)&=&\big(\xi(\gamma_{2,i}^-),\xi^*(\gamma_{2,i}^-)\big),\\
g_j\cdot \xi_j'(\gamma_{3,i}^-)&=&\xi(\gamma_{3,i}^-).
\end{eqnarray*}
Then define $\xi_j:=g_j\circ\xi_j'$ and $\Omega_j:=g_j\cdot\Omega_j'$. Note that in $\Dmc\backslash\PGL(3,\Rbbb)$, $[\xi_j,\Omega_j]=[\xi_j',\Omega_j']=\prod_{k=1}^j(\epsilon_j)_t[\xi,\Omega]$. 

By Proposition \ref{triple compute}, Proposition \ref{triple ratio and cross ratio facts}, and the fact that the triple ratio and cross ratio are invariant under projective transformations, we know that for any $k\in\Zbbb^+$ and for all $j\geq k$, 
\[\tau_{\xi_j,\xi_j^*}(a_k,b_k,c_k)=\left\{\begin{array}{ll}
\tau_{\xi,\xi^*}(a_k,b_k,c_k)+t&\text{if }T_k=\gamma\cdot\Delta_i\text{ for some }\gamma\in\pi_1(S)\\
\tau_{\xi,\xi^*}(a_k,b_k,c_k)-t&\text{if }T_k=\gamma\cdot\Delta_i'\text{ for some }\gamma\in\pi_1(S)\\
\end{array}\right.,\]
where $a_k,b_k,c_k$ are the vertices of $T_k$ so that $a_k<b_k<c_k<a_k$.
Also, if $T=\big\{\{a,b\},\{b,c\},\{c,a\}\big\}$ is a triangle so that $T\neq \gamma\cdot\Delta_i$ and $T\neq\gamma\cdot\Delta_i'$ for all $\gamma\in\pi_1(S)$, then for all $j\in\Zbbb^+$, we have that $\tau_{\xi_j,\xi_j^*}(a,b,c)=\tau_{\xi,\xi^*}(a,b,c)$. For the same reasons, 
\begin{eqnarray*}
\sigma_{\xi_j,\xi_j^*}(x,z_{x,y},z'_{x,y},y)&=&\sigma_{\xi,\xi^*}(x,z_{x,y},z'_{x,y},y)\,\,\,\text{ and }\\
\sigma_{\xi_j,\xi_j^*}(y,z'_{x,y},z_{x,y},x)&=&\sigma_{\xi,\xi^*}(y,z'_{x,y},z_{x,y},x)
\end{eqnarray*}
for any $\{x,y\}\in\Tmc$ that do not correspond to the three boundary curves of $P_i$.

Now, we will state and give a sketch of the proof of the main technical lemma we need to prove Theorem \ref{main theorem}.

\begin{lem}\label{technical lemma}
Let $\{x,y\}\in\widetilde{\Pmc}$ and let $z,z'$ be vertices of $\widetilde{\Tmc}$ so that $\{x,z\},\{y,z'\}\in\widetilde{\Tmc}$. Then the sequences 
\[\{\sigma_{\xi_j,\xi_j^*}(x,z,z',y)\}_{j=1}^\infty,\,\,\,\,\, \{\sigma_{\xi_j,\xi_j^*}(y,z',z,x)\}_{j=1}^\infty,\] 
\[\{\tau_{\xi_j,\xi_j^*}(x,z,y)\}_{j=1}^\infty,\,\,\,\,\,\{\tau_{\xi_j,\xi_j^*}(y,z',x)\}_{j=1}^\infty\] 
converge to positive real numbers.
\end{lem}

\begin{proof}
Assume that $\{x,y\}$ in $\Pmc$ corresponds to a boundary component of $P_i$ (otherwise, the lemma is true by the comments above). We will first focus on the sequence $\{\tau_{\xi_j,\xi_j^*}(x,z,y)\}_{j=1}^\infty$. If $\{x,z\}$ is not an edge in $\widetilde{P}_{i,j}$ for some $j=1,\dots,\infty$, then it is immediate that $\{\tau_{\xi_j,\xi_j^*}(x,z,y)\}_{j=1}^\infty$ is the constant sequence and therefore converges. Thus, assume that $\{x,z\}$ is an edge in $\widetilde{P}_{i,j}$ for some $j$.

Let $\gamma\in\pi_1(S)$ be the primitive group element with $x$ and $y$ as its repelling and attracting fixed points respectively. Then let $w\in\partial\pi_1(S)$ be the point so that $\{w,x\}\in\widetilde{\Tmc}$ and $w\in(z,\gamma\cdot z)_x$ (see Notation \ref{interval notation}). Similarly, let $w'\in\partial\pi_1(S)$ be the point so that $\{w',y\}\in\widetilde{\Tmc}$ and $w'\in (z',\gamma^{-1}\cdot z')_x$. For all $j=1,\dots,\infty$, let $\overline{g}_j$ be a group element so that 
\[\overline{g}_j\cdot \xi_j(x)=\xi(x),\,\,\,\,\overline{g}_j\cdot \xi_j(z)=\xi(z),\,\,\,\,\overline{g}_j\cdot \xi_j(w)=\xi(w),\] 
\[\overline{g}_j\cdot \xi_j^*(x)=\xi^*(x),\,\,\,\,\overline{g}_j\cdot \xi_j^*(z)=\xi^*(z),\]
and let $(\overline{\xi}_j,\overline{\xi^*_j}):=\overline{g}_j\cdot(\xi_j,\xi_j^*)$. By definition,
\[\big(\overline{\xi_j}(x),\overline{\xi_j^*}(x)\big)=\big(\xi(x),\xi^*(x)\big),\,\,\,\, \big(\overline{\xi_j}(z),\overline{\xi_j^*}(z)\big)=\big(\xi(z),\xi^*(z)\big)\,\,\,\,\text{ and }\,\,\,\, \overline{\xi_j}(w)=\xi(w),\]\
for all $j\in\Zbbb^+$.

\begin{figure}[ht]
\centering
\includegraphics[scale=0.5]{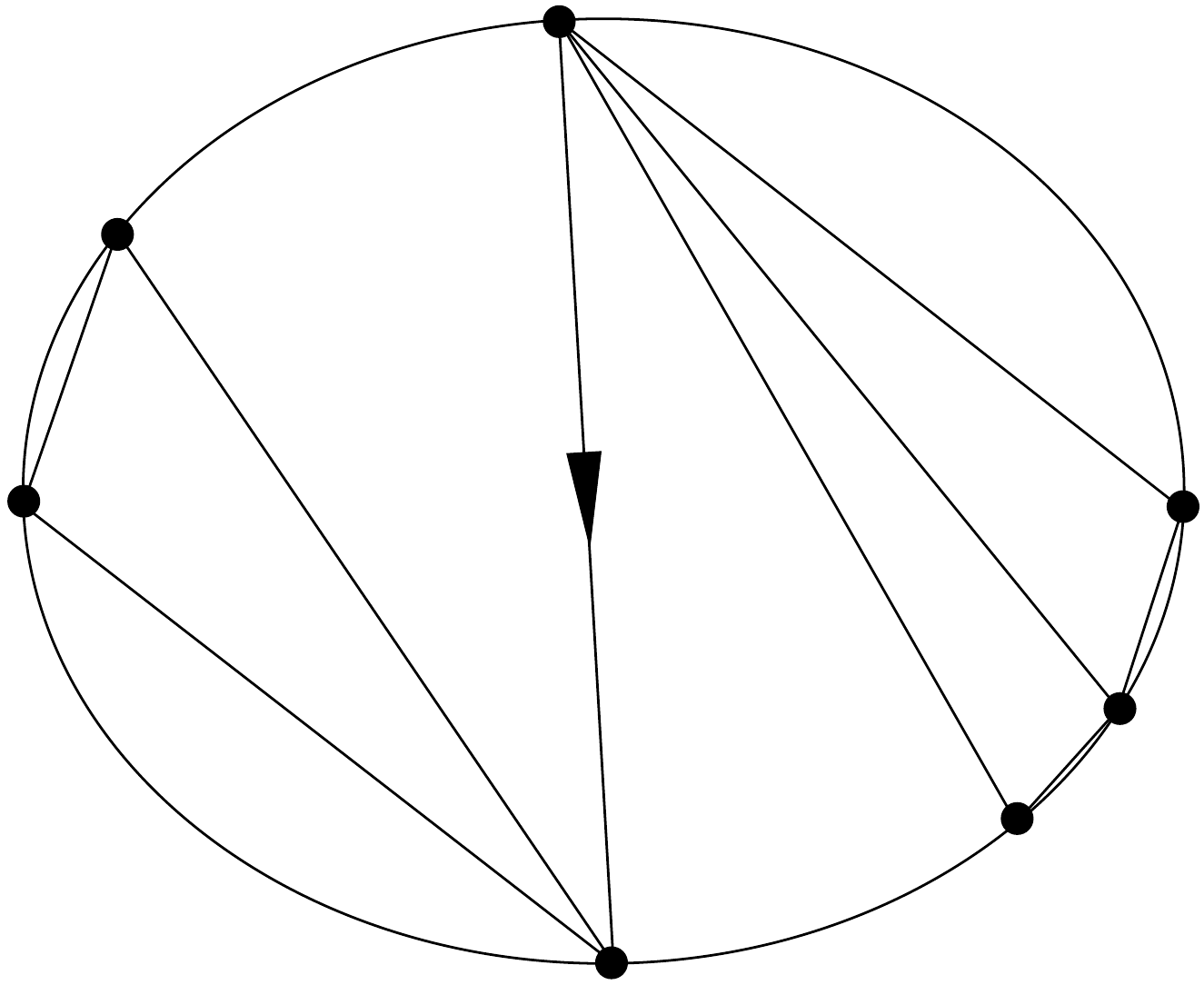}
\small
\put (-97, -4){$y$}
\put (-105, 157){$x$}
\put (0, 74){$z$}
\put (-197, 75){$z'$}
\put (-10, 42){$w$}
\put (-183, 120){$w'$}
\put (-26, 22){$\gamma\cdot z$}
\put (-107, 78){$\gamma$}
\put (-30, 74){$T$}
\put (-38, 50){$T'$}
\caption{Proof of Lemma \ref{technical lemma}.}\label{eruptionproof}
\end{figure}

Let $T:=\big\{\{x,z\},\{z,w\},\{w,x\}\big\}$ and let $T':=\big\{\{x,w\},\{w,\gamma\cdot z\},\{\gamma\cdot z,x\}\big\}$ (see Figure \ref{eruptionproof}). Consider the sequence of triangles 
\[\{T,T',\gamma\cdot T,\gamma\cdot T',\dots,\gamma^k\cdot T,\gamma^k\cdot T',\gamma^{k+1}\cdot T,\dots\}=\{T_{j_1},T_{j_2},\dots\},\]
where the right-hand side is written using the enumeration of $\Theta_i$. Then observe that $j_1<j_2<\dots$. Also, let $h_t$ be the projective transformation
\[h_t:=\left[\begin{array}{ccc}e^\frac{t}{3}&0&0\\0&e^{-\frac{t}{3}}&0\\0&0&1\end{array}\right]\left[\begin{array}{ccc}1&0&0\\0&e^{\frac{t}{3}}&0\\0&0&e^{-\frac{t}{3}}\end{array}\right]^{-1}=\left[\begin{array}{ccc}e^{\frac{t}{3}}&0&0\\0&e^{-\frac{2t}{3}}&0\\0&0&e^{\frac{t}{3}}\end{array}\right]\]
written in the basis $\{v_w,v_x,v_z\}$, where $[v_q]=\xi(q)$ for $q=w,x,z$. 

By the definition of $(\overline{\xi}_j,\overline{\xi^*_j})$, we see that 
\[\big(\overline{\xi_{j_1}}(p),\overline{\xi_{j_1}^*}(p)\big)=h_t\cdot\big(\xi(p),\xi^*(p)\big)\]
for all $p\in\{q\in\partial\pi_1(S):q=y\text{ or }q\text{ is a vertex of }T_{j_k}\text{ for some }k\geq 2\}$. (Note that $h_t$ fixes $\big(\xi(x),\xi^*(x)\big)$ and $\xi(w)$.) Furthermore, for all $k\geq j_1$, 
\[\overline{\xi_k^*}(w)=\overline{\xi^*_{j_1}}(w)\,\,\,\,\text{ and }\,\,\,\,\overline{\xi_k}(\gamma\cdot z)=\overline{\xi_{j_1}}(\gamma\cdot z).\]

Next, let $g$ the projective transformation that fixes the flag $\big(\overline{\xi_{j_1}}(x), \overline{\xi_{j_1}^*}(x)\big)=\big(\xi(x),\xi^*(x)\big)$, sends the flag $\big(\overline{\xi_{j_1}}(w),\overline{\xi_{j_1}^*}(w)\big)$ to $\big(\overline{\xi_{j_1}}(z),\overline{\xi_{j_1}^*}(z)\big)$ and sends the point $\overline{\xi_{j_1}}(\gamma\cdot z)$ to $\overline{\xi_{j_1}}(w)$. Again by the definition of $(\overline{\xi_j},\overline{\xi_j^*})$, we have that 
\[\big(\overline{\xi_{j_2}}(p),\overline{\xi_{j_2}^*}(p)\big)=g^{-1}h_{-t}g\cdot\big(\overline{\xi_{j_1}}(p),\overline{\xi_{j_1}^*}(p)\big)=g^{-1}h_t^{-1}gh_t\cdot\big(\xi(p),\xi^*(p)\big)\]
for all $p\in\{q\in\partial\pi_1(S):q=y\text{ or }q\text{ is a vertex of }T_{j_k}\text{ for some }k\geq 3\}$. (Note that $g^{-1}h_{-t}g$ fixes $\big(\xi(x),\xi^*(x)\big)$ and $\overline{\xi_{j_1}}(\gamma\cdot z)$.)  Furthermore, for all $k\geq j_2$,
\[\overline{\xi_k^*}(\gamma\cdot z)=\overline{\xi_{j_2}^*}(\gamma\cdot z)\,\,\,\,\text{ and }\,\,\,\,\overline{\xi_k}(\gamma\cdot w)=\overline{\xi_{j_2}}(\gamma\cdot w).\]
Let $u_t:=g^{-1}h_t^{-1}gh_t$, and observe that $u_t$ is a unipotent projective transformation that fixes $\big(\xi(x),\xi^*(x)\big)$. 

Let $r$ be the projective transformation that fixes the flag $\big(\overline{\xi_{j_2}}(x),\overline{\xi_{j_2}^*}(x)\big)$, sends the flag $\big(\overline{\xi_{j_2}}(z),\overline{\xi_{j_2}^*}(z)\big)$ to $\big(\overline{\xi_{j_2}}(\gamma\cdot z),\overline{\xi_{j_2}^*}(\gamma\cdot z)\big)$ and sends the point $\overline{\xi_{j_2}}(w)$ to $\overline{\xi_{j_2}}(\gamma\cdot w)$. It is easy to see that $r$ is diagonalizable with eigenvalues having pairwise distinct absolute values. As a consequence of Proposition \ref{finite parameterization}, for all $k\geq j_3$, we have $\overline{\xi_k^*}(\gamma\cdot w)=r\cdot \overline{\xi_{j_2}^*}(\gamma\cdot w)$. For the same reasons,
\[\big(\overline{\xi_{j_4}}(p),\overline{\xi_{j_4}^*}(p)\big)=ru_tr^{-1}\cdot\big(\overline{\xi_{j_2}}(p),\overline{\xi_{j_2}^*}(p)\big)=ru_tr^{-1}u_t\cdot\big(\xi(p),\xi^*(p)\big)\]
for all $p\in\{q\in\partial\pi_1(S):q=y\text{ or }q\text{ is a vertex of }T_{j_k}\text{ for some }k\geq 5\}$. In particular, $r$ fixes the flag $\big(\overline{\xi_{j_4}}(x),\overline{\xi_{j_4}^*}(x)\big)=\big(\xi(x),\xi^*(x)\big)$, sends the flag $\big(\overline{\xi_{j_4}}(\gamma\cdot z),\overline{\xi_{j_4}^*}(\gamma\cdot z)\big)$ to $\big(\overline{\xi_{j_4}}(\gamma^2\cdot z),\overline{\xi_{j_4}^*}(\gamma^2\cdot z)\big)$ and sends the point $\overline{\xi_{j_4}}(\gamma\cdot w)$ to $\overline{\xi_{j_4}}(\gamma^2\cdot w)$.

By iterating this process, we then see that 
\[\lim_{j\to\infty}\big(\overline{\xi_j}(y),\overline{\xi_j^*}(y)\big)=\prod_{k=0}^\infty r^ku_tr^{-k}\cdot\big(\xi(y),\xi^*(y)\big).\]
Here, the product of projective transformations $\prod_{k=0}^\infty r^ku_tr^{-k}$ converges because $u_t$ is unipotent, $r$ is diagonalizable with pairwise distinct eigenvalues, and the repelling flag of $r$ is exactly the fixed flag of $u_t$. 

More explicitly, if we let $v_1,v_2,v_3\in\Rbbb^3$ be eigenvectors of $r$, enumerated in increasing order of the absolute values of the corresponding eigenvalues, then
\[r=\left[\begin{array}{ccc}\alpha&0&0\\0&\beta&0\\0&0&\gamma\end{array}\right]\,\,\,\,\text{ and }\,\,\,\,u_t=\left[\begin{array}{ccc}1&a&b\\0&1&c\\0&0&1\end{array}\right]\]
in the basis $\{v_1,v_2,v_3\}$, with $\alpha<\beta<\gamma$. It is then an elementary computation to show that
\[\prod_{k=0}^\infty r^ku_tr^{-k}=\left[\begin{array}{ccc}1&\frac{a\beta}{\beta-\alpha}&\frac{b\gamma}{\gamma-\alpha}+\frac{ac\alpha\gamma}{(\beta-\alpha)(\gamma-\alpha)}\\0&1&\frac{c\gamma}{\gamma-\beta}\\0&0&1\end{array}\right].\]
In particular, $\prod_{k=0}^\infty r^ku_tr^{-k}$ is a unipotent projective transformation that fixes the flag $\big(\xi(x),\xi^*(x)\big)$. 

Since $\lim_{j\to\infty}\big(\overline{\xi_j}(y),\overline{\xi_j^*}(y)\big)$ is the image of $\big(\xi(y),\xi^*(y)\big)$ under a unipotent projective transformation that fixes $\big(\xi(x),\xi^*(x)\big)$, it follows from the transversality of $\big(\xi(y),\xi^*(y)\big)$ and $\big(\xi(x),\xi^*(x)\big)$ that
$\lim_{j\to\infty}\big(\overline{\xi_j}(y),\overline{\xi_j^*}(y)\big)$ and $\lim_{j\to\infty}\big(\overline{\xi_j}(x),\overline{\xi_j^*}(x)\big)$ are also transverse. 

On the other hand, the projective line segment $[\xi(x),\xi(w)]$ and the projective lines $\xi^*(x)$, $\xi^*(w)$ determine two triangles in $\Rbbb\Pbbb^2$, one of which contains $\overline{\xi_j}(z)$ for all $j$ and the other contains $\overline{\xi_j}(y)$ for all $j$. It is also easy to see that $\overline{\xi_j^*}(z)=\xi^*(z)$ does not intersect the closure of the latter triangle for all $j$, so $\lim_{j\to\infty}\overline{\xi_j}(y)$ does not lie in $\lim_{j\to\infty}\overline{\xi_j^*}(z)=\xi^*(z)$. Performing the same argument in $(\Rbbb\Pbbb^2)^*$ proves that $\lim_{j\to\infty}\overline{\xi_j^*}(y)$ does not contain $\lim_{j\to\infty}\overline{\xi_j}(z)$. 

This thus proves that the triple of flags 
\[\lim_{j\to\infty}\big(\overline{\xi_j}(x),\overline{\xi_j^*}(x)\big),\,\,\,\, \lim_{j\to\infty}\big(\overline{\xi_j}(y),\overline{\xi_j^*}(y)\big),\,\,\,\, \lim_{j\to\infty}\big(\overline{\xi_j}(z),\overline{\xi_j^*}(z)\big)\]
is pairwise transverse. As such, the triple ratio of this triple must be either positive or negative. Since they are limits of triples of pairwise transverse flags whose triple ratios are positive, and the triple ratio varies continuously on $\Fmc_3^+$, we can conclude that the triple ratio of this triple is positive. It follows immediately that the sequence $\{\tau_{\xi_j,\xi_j^*}(y,z,x)\}_{j=1}^\infty$ converges to a positive real number. The same argument, using the points $w'$ and $z'$ in place of $w$ and $z$, proves that $\{\tau_{\xi_j,\xi_j^*}(x,z',y)\}_{j=1}^\infty$ also converges to a positive real number.

Now, we deal with the convergence of the shear parameters under the eruption flow, thus we consider the sequences $\{\sigma_{\xi_j,\xi_j^*}(x,z,z',y)\}_{j=1}^\infty$ and $\{\sigma_{\xi_j,\xi_j^*}(y,z',z,x)\}_{j=1}^\infty$. Suppose that only one of $\{x,z\}$ or $\{y,z'\}$ lies in $\widetilde{P}_{i,j}$ for some $j$. If $\{y,z'\}$ does not lie in $\widetilde{P}_{i,j}$ for some $j$, it follows from the above argument that 
\[\lim_{j\to\infty}\big(\overline{\xi_j}(z'),\overline{\xi_j^*}(z')\big)=\prod_{k=0}^\infty r^ku_tr^{-k}\cdot\big(\xi(z'),\xi^*(z')\big),\]
which allows us to conclude that the quadruple of flags
\[\lim_{j\to\infty}\big(\overline{\xi_j}(x),\overline{\xi_j^*}(x)\big),\,\,\,\, \lim_{j\to\infty}\big(\overline{\xi_j}(y),\overline{\xi_j^*}(y)\big), \,\,\,\,\lim_{j\to\infty}\big(\overline{\xi_j}(z),\overline{\xi_j^*}(z)\big),\,\,\,\,\lim_{j\to\infty}\big(\overline{\xi_j}(z'),\overline{\xi_j^*}(z')\big)\]
is pairwise transverse. A continuity argument then allows us to conclude that both $\{\sigma_{\xi_j,\xi_j^*}(x,z,z',y)\}_{j=1}^\infty$ and $\{\sigma_{\xi_j,\xi_j^*}(y,z',z,x)\}_{j=1}^\infty$ converge to positive real numbers. The same argument can be used in the case when $\{x,z\}$ does not lie in $\widetilde{P}_{i,j}$ for some $j$. 

To deal with the case when both $\{x,z\}$ and $\{y,z'\}$ lie in $\widetilde{P}_{i,j}$ for some $j$, we need to slightly modify the argument given above. We leave this modification to the reader. 
\end{proof}

By a similar argument, one can also prove the following lemma.

\begin{lem}\label{neighboring invariants}
Let $\{y_1,y_2\}\in\widetilde{\Qmc}$. For $i=1,2$, let $\gamma_i\in\pi_1(S)$ be the primitive group element whose repelling fixed point is $y_i$, and let $x_i$ be the attracting fixed point of $\gamma_i$. Also, let $w_i$ be the vertex of $\widetilde{\Tmc}$ with the property that there are triangles $T_i,T_i'\in\Theta_{\widetilde{\Tmc}}$ so that $\{y_1,y_2\}$, $\{y_i,w_i\}$ are edges of $T_i$ and $\gamma_i\cdot\{y_1,y_2\}$, $\{y_i,w_i\}$ are edges of $T_i'$ (see Figure \ref{eruptionproof2}). Then the sequences 
\[\{\sigma_{\xi_j,\xi_j^*}(y_1,x_1,x_2,y_2)\}_{j=1}^\infty,\,\,\,\,\,\{\sigma_{\xi_j,\xi_j^*}(y_1,x_1,w_2,y_2)\}_{j=1}^\infty,\,\,\,\,\, \{\sigma_{\xi_j,\xi_j^*}(y_1,w_1,x_2,y_2)\}_{j=1}^\infty,\] 
converge to positive real numbers. 
\end{lem}

\begin{figure}[ht]
\centering
\includegraphics[scale=0.5]{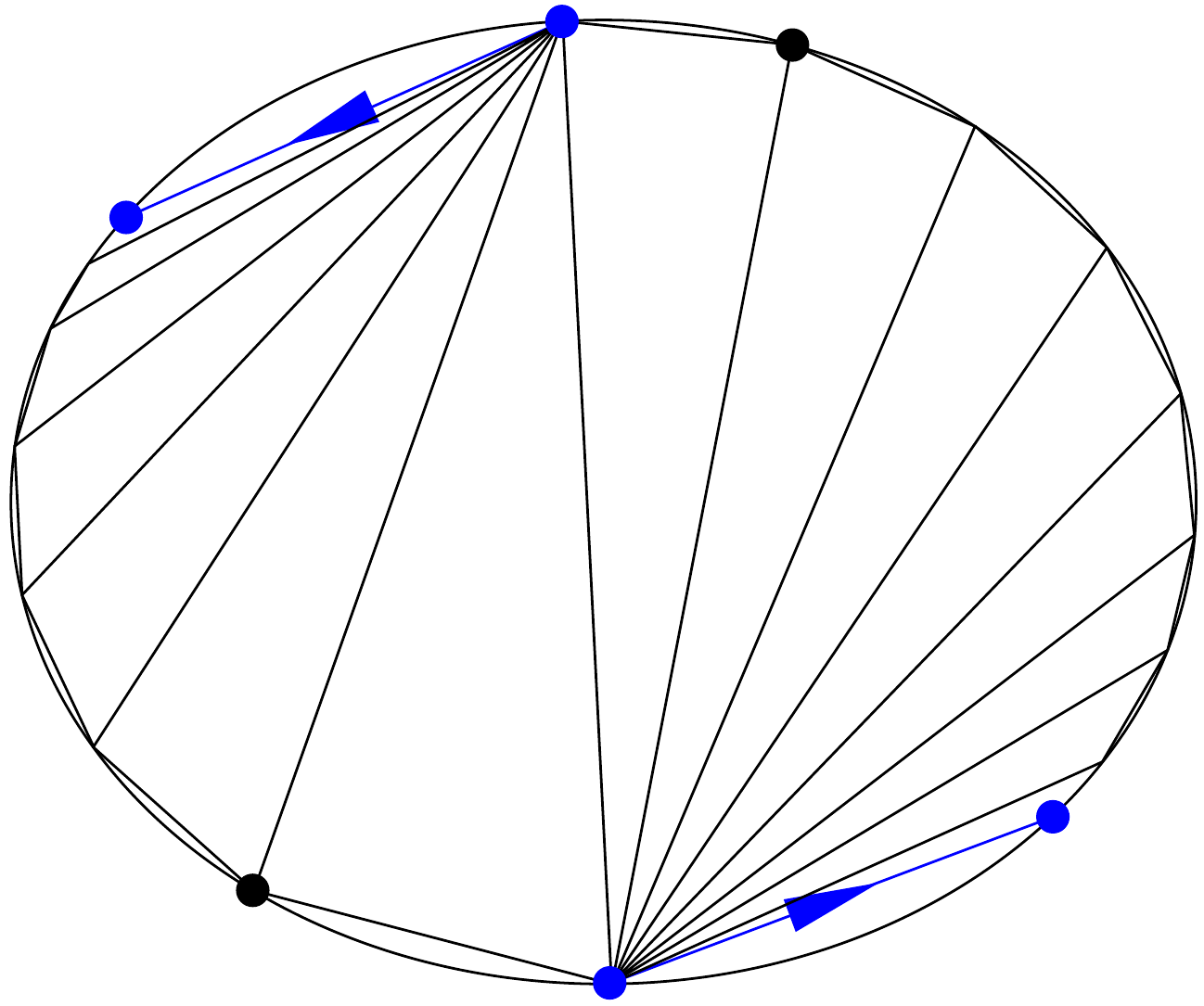}
\small
\put (-103, 159){$y_1$}
\put (-96, -4){$y_2$}
\put (-180, 122){$x_1$}
\put (-20, 28){$x_2$}
\put (-154, 10){$w_1$}
\put (-68, 154){$w_2$}
\put (-120, 47){$T_1$}
\put (-153, 54){$T_1'$}
\put (-87, 120){$T_2$}
\put (-65, 110){$T_2'$}
\caption{Lemma \ref{neighboring invariants}.}\label{eruptionproof2}
\end{figure}

By Theorem \ref{Bonahon-Dreyer}, there is some $[f_t,\Sigma_t]\in\Cmc(S)$ so that the following hold:
\begin{itemize}
\item For all $e=[x,y]\in\Qmc$ and for all $e=[x,y]\in\Pmc$ that does not correspond to a boundary component of $P_i$,
\[\sigma_{e,x}[f_t,\Sigma_t]=\sigma_{e,x}[f,\Sigma]\,\,\,\,\text{ and }\,\,\,\,\sigma_{e,y}[f_t,\Sigma_t]=\sigma_{e,y}[f,\Sigma]\]
\item For all $e=[x,y]\in\Pmc$ that correspond to a boundary component of $P_i$, 
\[\sigma_{e,x}[f_t,\Sigma_t]=\lim_{j\to\infty}\sigma_{\xi_j,\xi_j^*}(x,z_{x,y},z_{x,y}',y)\text{ and }\sigma_{e,y}[f_t,\Sigma_t]=\lim_{j\to\infty}\sigma_{\xi_j,\xi_j^*}(y,z'_{x,y},z_{x,y},x)\]
\item $\tau_T[f_t,\Sigma_t]=\tau_T[f,\Sigma]$ for all $T\in\Theta_\Tmc$ so that $T\neq [\Delta_i],[\Delta'_i]$,
\item $\tau_{[\Delta_i]}[f_t,\Sigma_t]=\tau_{[\Delta_i]}[f,\Sigma]+t\text{ and }\tau_{[\Delta_i']}[f_t,\Sigma_t]=\tau_{[\Delta_i']}[f,\Sigma]-t$.
\end{itemize}
Let $(\xi_t,\Omega_t)$ be the representative of $[\xi_t,\Omega_t]\in\PGL(3,\Rbbb)\backslash\Dmc$ so that $[\xi_t,\Omega_t]$ corresponds to $[f_t,\Sigma_t]\in\Cmc(S)$, and
\[\big(\xi_t(\gamma_{1,i}^-),\xi_t^*(\gamma_{1,i}^-)\big)=\big(\xi(\gamma_{1,i}^-),\xi^*(\gamma_{1,i}^-)\big),\,\,\,\, \big(\xi_t(\gamma_{2,i}^-),\xi_t^*(\gamma_{2,i}^-)\big)=\big(\xi(\gamma_{2,i}^-),\xi^*(\gamma_{2,i}^-)\big),\]
\[\text{and }\,\,\,\,\xi_t(\gamma_{3,i}^-)=\xi(\gamma_{3,i}^-).\]

\begin{lem}\label{convergence of vertices}
Let $x_0\in\partial\pi_1(S)$ be any vertex of $\widetilde{\Tmc}$. Then 
\[\lim_{j\to\infty}\big(\xi_j(x_0),\xi_j^*(x_0)\big)=\big(\xi_t(x_0),\xi_t^*(x_0)\big).\]
\end{lem}

\begin{proof}
First, one can compute that the following equalities hold:
\begin{itemize}
\item Let $x,y,z,z'$ be as defined in the statement of Lemma \ref{technical lemma}. Then
\[\lim_{j\to\infty}\tau_{\xi_j,\xi_j^*}(x,z,y)=\tau_{\xi_t,\xi_t^*}(x,z,y),\]
\[\lim_{j\to\infty}\tau_{\xi_j,\xi_j^*}(y,z',x)=\tau_{\xi_t,\xi_t^*}(y,z',x).\]
\item Let $x_1,x_2,y_1,y_2,w_2,w_2$ be as defined in the statement of Lemma \ref{neighboring invariants}. Then
\[\lim_{j\to\infty}\sigma_{\xi_j,\xi_j^*}(y_1,x_1,x_2,y_2)=\sigma_{\xi_t,\xi_t^*}(y_1,x_1,x_2,y_2),\]
\[\lim_{j\to\infty}\sigma_{\xi_j,\xi_j^*}(y_1,x_1,w_2,y_2)=\sigma_{\xi_t,\xi_t^*}(y_1,x_1,w_2,y_2),\]
\[\lim_{j\to\infty}\sigma_{\xi_j,\xi_j^*}(y_1,w_1,x_2,y_2)=\sigma_{\xi_t,\xi_t^*}(y_1,w_1,x_2,y_2).\]
\end{itemize}

If $x_0=\gamma_{1,i}^-$ or $x_0=\gamma_{2,i}^-$, then by definition of $\xi_j$ and $\xi_t$, $\lim_{j\to\infty}\big(\xi_j(x_0),\xi_j^*(x_0)\big)=\big(\xi_t(x_0),\xi_t^*(x_0)\big)$. For the same reason, if $x_0=\gamma_{3,i}^-$, then $\lim_{j\to\infty}\xi_j(x_0)=\xi_t(x_0)$. On the other hand, $\lim_{j\to\infty}\xi_j^*(x_0)=\xi_t^*(x_0)$ because $\lim_{j\to\infty}\tau_{\xi_j,\xi_j^*}(\gamma_{1,i}^-,\gamma_{3,i}^-,\gamma_{2,i}^-)=\tau_{\xi_t,\xi_t^*}(\gamma_{1,i}^-,\gamma_{3,i}^-,\gamma_{2,i}^-)$.

Now, suppose that $x_0$ is not any of $\gamma_{1,i}^-,\gamma_{2,i}^-,\gamma_{3,i}^-$. Assume without loss of generality that $x_0\in(\gamma_{1,i}^-,\gamma_{2,i}^-)_{\gamma_{3,i}^-}$ (see Notation \ref{interval notation}), then $\overline{\Emc}_{x_0,\gamma_{3,i}^-}$ is finite and non-empty. ($\overline{\Emc}_{x_0,\gamma_{3,i}^-}$ was defined using $\widetilde{\Tmc}$ in Section \ref{ideal triangulation}.) By Lemma \ref{neighboring invariants} and the observation stated prior to Lemma \ref{technical lemma}, we know that for every triple of consecutive edges $\{z',y\}$, $\{x,y\}$, $\{z,x\}$ of $\overline{\Emc}_{x_0,\gamma_{3,i}^-}$ in this order,
\[\lim_{j\to\infty}\sigma_{\xi_j,\xi_j^*}(x,z,z',y)=\sigma_{\xi_t,\xi_t^*}(x,z,z',y),\,\,\,\,\,\,\lim_{j\to\infty}\sigma_{\xi_j,\xi_j^*}(y,z',z,x)=\sigma_{\xi_t,\xi_t^*}(y,z',z,x),\]
\[\lim_{j\to\infty}\tau_{\xi_j,\xi_j^*}(x,y,z)=\sigma_{\xi_t,\xi_t^*}(x,y,z),\,\,\,\,\,\,\lim_{j\to\infty}\tau_{\xi_j,\xi_j^*}(y,z',x)=\sigma_{\xi_t,\xi_t^*}(y,z',x).\]
We can then apply Proposition \ref{finite convergence} to conclude that $\lim_{j\to\infty}\big(\xi_j(x_0),\xi_j^*(x_0)\big)=\big(\xi_t(x_0),\xi_t^*(x_0)\big)$.
\end{proof}

\begin{proof}[Proof of Theorem \ref{main theorem}]
By a standard density argument, Lemma \ref{convergence of vertices} implies that for any $x\in\partial\pi_1(S)$, 
\[\lim_{j\to\infty}\big(\xi_j(x),\xi_j^*(x)\big)=\big(\xi_t(x),\xi_t^*(x)\big).\]
Hence, $(E_i)_t[f,\Sigma]=[f_t,\Sigma_t]$. Since $[f,\Sigma]\in\Cmc(S)$ was arbitrary, this proves that $(E_i)_t$ is well-defined. From the proof of Lemma \ref{technical lemma}, it is clear that 
\[\lim_{j\to\infty}\sigma_{\xi_j,\xi_j^*}(x,z_{x,y},z_{x,y}',y)\,\,\,\,\text{ and }\,\,\,\,\lim_{j\to\infty}\sigma_{\xi_j,\xi_j^*}(y,z'_{x,y},z_{x,y},x)\] 
vary smoothly with $t$. Thus, by the way we chose $[f_t,\Sigma_t]$, the Bonahon-Dreyer parameters for $[f_t,\Sigma_t]$ vary smoothly with $t$. The smoothness of $(E_i)_t$ then follows from the smoothness of the Bonahon-Dreyer parameterization of $\Cmc(S)$. Finally, to see the independence of $(E_i)_t$ from the choices made to obtain the enumeration of $\Theta_i$, simply observe that the Bonahon-Dreyer parameters for $[f_t,\Sigma_t]$ does not depend on any such choice. This finishes the proof of Theorem \ref{main theorem}.
\end{proof}

\subsection{Internal bulging flows} 
Now we define the internal bulging flows on $\Cmc(S)$, which arise from bulging flows along (internal) edges of triangulations $\Tmc$, which do not correspond to simple closed curves in the pants decomposition $\Pmc$. 

Orient the edges $e_{1,i}$, $e_{2,i}$, $e_{3,i}$ so that they have backward endpoints $\gamma_{1,i}^-$, $\gamma_{2,i}^-$, $\gamma_{3,i}^-$ and forward endpoints $\gamma_{2,i}^-$, $\gamma_{3,i}^-$, $\gamma_{1,i}^-$ respectively. ($e_{1,i}$, $e_{2,i}$, $e_{3,i}$ were defined at the beginning of Section~\ref{deform_convex}.) This induces an orientation on all the edges in $\widetilde{\Qmc}_i$. For each $e_j\in\widetilde{\Qmc}_i$, let $e_j^+$ and $e_j^-$ denote its forward and backward endpoints respectively. For any $t\in\Rbbb$ and any $j\in\Zbbb^+$, let 
\[(\beta_j)_t:=(\beta_{e_j^+,e_j^-})_t:\PGL(3,\Rbbb)\backslash\Dmc\to\PGL(3,\Rbbb)\backslash\Dmc.\]
Then define $(I_i)_t:=\prod_{j=1}^\infty\beta_j(t):\Cmc(S)\to\Cmc(S)$. 

\begin{definition}
The flow $(I_i)_t$ on $\Cmc(S)$ defined above is the \emph{internal bulging flow} associated to the pair of pants $P_i\subset S$
\end{definition}

An analogous argument as the one given in the proof of Theorem \ref{main theorem} gives the following statement. 

\begin{thm}\label{shear}
For all $i=1,\dots,2g-2$, $(I_i)_t$ is well-defined, smooth, and does not depend on any of the choices we made to obtain the enumeration of $\widetilde{\Qmc}_i$.  
\end{thm}

It is important to remark here that if we try to use the elementary shearing flow $\psi_{e_j^+,e_j^-}$ in place of $\beta_{e_j^+,e_j^-}$, then we do not get a well-defined flow on $\Cmc(S)$. The key point here is that if we do so, then the analogous projective transformation to $u_t$ defined in the proof of Lemma \ref{technical lemma} will not be unipotent.

\bibliographystyle{amsalpha}
\bibliography{ref}

\providecommand{\bysame}{\leavevmode\hbox to3em{\hrulefill}\thinspace}
\providecommand{\MR}{\relax\ifhmode\unskip\space\fi MR }
\providecommand{\MRhref}[2]{%
  \href{http://www.ams.org/mathscinet-getitem?mr=#1}{#2}
}
\providecommand{\href}[2]{#2}
\begin{thebibliography}{Zha15b}

\bibitem[BD]{BonahonDreyer2}
Francis Bonahon and Guillaume Dreyer, \emph{Hitchin characters and geodesic
  laminations}, Preprint https://arxiv.org/abs/1410.0729.

\bibitem[BD14]{BonahonDreyer1}
\bysame, \emph{Parameterizing {H}itchin components}, Duke Math. J. \textbf{163}
  (2014), no.~15, 2935--2975. \MR{3285861}

\bibitem[Ben60]{Benzecri}
Jean-Paul Benz\'ecri, \emph{Sur les vari\'et\'es localement affines et
  localement projectives}, Bull. Soc. Math. France \textbf{88} (1960),
  229--332. \MR{0124005}

\bibitem[BK]{BonahonKim}
Francis Bonahon and Inkang Kim, \emph{The goldman and fock-goncharov
  coordinates for convex projective structures on surfaces}, Preprint
  https://arxiv.org/abs/1607.03650.

\bibitem[CG93]{ChoiGoldman}
Suhyoung Choi and William~M. Goldman, \emph{Convex real projective structures
  on closed surfaces are closed}, Proc. Amer. Math. Soc. \textbf{118} (1993),
  no.~2, 657--661. \MR{1145415}

\bibitem[FG06]{FockGoncharov}
Vladimir Fock and Alexander Goncharov, \emph{Moduli spaces of local systems and
  higher {T}eichm\"uller theory}, Publ. Math. Inst. Hautes \'Etudes Sci.
  (2006), no.~103, 1--211. \MR{2233852}

\bibitem[FG07]{FockGoncharov_convex}
Vladimir~V. Fock and Alexander~B. Goncharov, \emph{Moduli spaces of convex
  projective structures on surfaces}, Adv. Math. \textbf{208} (2007), no.~1,
  249--273.

\bibitem[Gol]{Goldman_bulging}
William~M. Goldman, \emph{Bulging deformations of convex
  $\mathbb{R}\mathbb{P}^2$-manifolds}, Preprint
  https://arxiv.org/pdf/1302.0777v1.pdf.

\bibitem[Gol86]{Goldman_twist}
\bysame, \emph{Invariant functions on {L}ie groups and {H}amiltonian flows of
  surface group representations}, Invent. Math. \textbf{85} (1986), no.~2,
  263--302. \MR{846929}

\bibitem[Gol90]{Goldman_convex}
\bysame, \emph{Convex real projective structures on compact surfaces}, J.
  Differential Geom. \textbf{31} (1990), no.~3, 791--845. \MR{1053346}

\bibitem[Hit92]{Hitchin}
Nigel~J. Hitchin, \emph{Lie groups and {T}eichm\"uller space}, Topology
  \textbf{31} (1992), no.~3, 449--473.

\bibitem[Ker83]{Kerckhoff}
Steven~P. Kerckhoff, \emph{The {N}ielsen realization problem}, Ann. of Math.
  (2) \textbf{117} (1983), no.~2, 235--265. \MR{690845}

\bibitem[Kui54]{Kuiper}
N.~H. Kuiper, \emph{On convex locally-projective spaces}, Convegno
  {I}nternazionale di {G}eometria {D}ifferenziale, {I}talia, 1953, Edizioni
  Cremonese, Roma, 1954, pp.~200--213. \MR{0063115}

\bibitem[Lab07]{Labourie_convex}
Fran\c{c}ois Labourie, \emph{Flat projective structures on surfaces and cubic
  holomorphic differentials}, Pure Appl. Math. Q. \textbf{3} (2007), no.~4,
  Special Issue: In honor of Grigory Margulis. Part 1, 1057--1099. \MR{2402597}

\bibitem[Li16]{Li}
Qiongling Li, \emph{Teichm\"uller space is totally geodesic in {G}oldman
  space}, Asian J. Math. \textbf{20} (2016), no.~1, 21--46. \MR{3460757}

\bibitem[Lof01]{Loftin_convex}
John~C. Loftin, \emph{Affine spheres and convex {$\Bbb{RP}^n$}-manifolds},
  Amer. J. Math. \textbf{123} (2001), no.~2, 255--274. \MR{1828223}

\bibitem[Thu86]{Thurston_earthquake}
William~P. Thurston, \emph{Earthquakes in two-dimensional hyperbolic geometry},
  Low-dimensional topology and {K}leinian groups ({C}oventry/{D}urham, 1984),
  London Math. Soc. Lecture Note Ser., vol. 112, Cambridge Univ. Press,
  Cambridge, 1986, pp.~91--112. \MR{903860}

\bibitem[Zha15a]{Zhang_convex}
Tengren Zhang, \emph{The degeneration of convex {$\Bbb{RP}^2$} structures on
  surfaces}, Proceedings of the London Mathematical Society \textbf{111}
  (2015), no.~5, 967--1012.

\bibitem[Zha15b]{Zhang_internal}
\bysame, \emph{Degeneration of {H}itchin representations along internal
  sequences}, Geom. Funct. Anal. \textbf{25} (2015), no.~5, 1588--1645.
  \MR{3426063}

\end{thebibliography}

\end{document}